\documentclass[oneside,british,english]{amsart}
\usepackage[T1]{fontenc}
\usepackage[latin9]{inputenc}
\usepackage{verbatim}
\usepackage{amsthm}
\usepackage{amstext}
\usepackage{amssymb}
\usepackage[all]{xy}

\makeatletter
\numberwithin{equation}{section}
\numberwithin{figure}{section}
\theoremstyle{plain}
\newtheorem{thm}{\protect\theoremname}[section]
  \theoremstyle{plain}
  \newtheorem{prop}[thm]{\protect\propositionname}
  \theoremstyle{definition}
  \newtheorem{defn}[thm]{\protect\definitionname}
  \theoremstyle{plain}
  \newtheorem{lem}[thm]{\protect\lemmaname}
  \theoremstyle{remark}
  \newtheorem{claim}[thm]{\protect\claimname}
  \theoremstyle{plain}
  \newtheorem{cor}[thm]{\protect\corollaryname}
  \theoremstyle{definition}
  \newtheorem*{problem*}{\protect\problemname}

\usepackage{alex}

\usepackage[all]{xy}
\SelectTips{cm}{10}
\makeatletter
\newcommand{\xyC}[1]{%
\makeatletter
\xydef@\xymatrixcolsep@{#1}
\makeatother
} 

\newcounter{quotecount}

\newcommand{\MyQuote}[1]{\vspace{6pt}
\noindent($*$)\hfill
\parbox{0.93\textwidth}{#1}\\[6pt]}

\makeatother

\usepackage{babel}
  \addto\captionsbritish{\renewcommand{\claimname}{Claim}}
  \addto\captionsbritish{\renewcommand{\corollaryname}{Corollary}}
  \addto\captionsbritish{\renewcommand{\definitionname}{Definition}}
  \addto\captionsbritish{\renewcommand{\lemmaname}{Lemma}}
  \addto\captionsbritish{\renewcommand{\problemname}{Problem}}
  \addto\captionsbritish{\renewcommand{\propositionname}{Proposition}}
  \addto\captionsbritish{\renewcommand{\theoremname}{Theorem}}
  \addto\captionsenglish{\renewcommand{\claimname}{Claim}}
  \addto\captionsenglish{\renewcommand{\corollaryname}{Corollary}}
  \addto\captionsenglish{\renewcommand{\definitionname}{Definition}}
  \addto\captionsenglish{\renewcommand{\lemmaname}{Lemma}}
  \addto\captionsenglish{\renewcommand{\problemname}{Problem}}
  \addto\captionsenglish{\renewcommand{\propositionname}{Proposition}}
  \addto\captionsenglish{\renewcommand{\theoremname}{Theorem}}
  \providecommand{\claimname}{Claim}
  \providecommand{\corollaryname}{Corollary}
  \providecommand{\definitionname}{Definition}
  \providecommand{\lemmaname}{Lemma}
  \providecommand{\problemname}{Problem}
  \providecommand{\propositionname}{Proposition}
\providecommand{\theoremname}{Theorem}

\begin{document}
\title[Similarity and commutators of matrices over PIDs]{Similarity and commutators of matrices over principal ideal rings}

\author{Alexander Stasinski}
\begin{abstract}
We prove that if $R$ is a principal ideal ring and $A\in\M_{n}(R)$
is a matrix with trace zero, then $A$ is a commutator, that is, $A=XY-YX$
for some $X,Y\in\M_{n}(R)$. This generalises the corresponding result
over fields due to Albert and Muckenhoupt, as well as that over $\Z$
due to Laffey and Reams, and as a by-product we obtain new simplified
proofs of these results. We also establish a normal form for similarity
classes of matrices over PIDs, generalising a result of Laffey and
Reams. This normal form is a main ingredient in the proof of the result
on commutators.
\end{abstract}

\address{Department of Mathematical Sciences, Durham University, South Rd,
Durham, DH1 3LE, UK}

\email{alexander.stasinski@durham.ac.uk}

\maketitle

\section{Introduction}

Let $R$ denote an arbitrary ring. If a matrix $A\in\M_{n}(R)$ is
a commutator, that is, if $A=[X,Y]=XY-YX$ for some $X,Y\in\M_{n}(R)$,
then $A$ must have trace zero. The problem of when the converse holds
goes back at least to Shoda \cite{Shoda} who showed in 1937 that
if $K$ is a field of characteristic zero, then every $A\in\M_{n}(K)$
with trace zero is a commutator. Shoda's argument fails in positive
characteristic, but Albert and Muckenhoupt \cite{Albert-Muckenhoupt}
found another argument valid for all fields. The first result for
rings which are not fields was obtained by Lissner \cite{Lissner}
who proved that if $R$ is a principal ideal domain (PID) then every
$A\in\M_{2}(R)$ with trace zero is a commutator. A motivation for
Lissner's work was the relation with a special case of Serre's problem
on projective modules over polynomial rings, nowadays known as the
Quillen-Suslin theorem (see \cite[Sections~1-2]{Lissner}). Lissner's
result on commutators in $\M_{2}(R)$ for $R$ a PID was rediscovered
by Vaserstein \cite{Vaserstein/87} and Rosset and Rosset \cite{Rosset},
respectively. Vaserstein also formulated the problem of whether every
$A\in\M_{n}(\Z)$ with trace zero is a commutator for $n\geq3$ (see
\cite[Section~5]{Vaserstein/87}). A significant breakthrough was
made by Laffey and Reams \cite{Laffey-Reams} who settled Vaserstein's
problem in the affirmative. However, their proofs involve steps which
are special to the ring of integers $\Z$ and do not generalise to
other rings in any straightforward way. The most crucial step of this
kind is an appeal to Dirichlet's theorem on primes in arithmetic progressions.
The analogue of Dirichlet's theorem, although true in the ring $\F_{q}[x]$,
fails for other Euclidean domains such as $\C[x]$ or discrete valuation
rings. Nevertheless, in \cite{Laffey-notes} Laffey asked whether
any matrix with trace zero over a Euclidean domain is a commutator.
Until now this appears to have remained an open problem even for $n=3$,
except for the cases where $R$ is a field or $\Z$.

In the present paper we answer Laffey's question by proving that if
$R$ is any PID and $A\in\M_{n}(R)$ is a matrix with trace zero,
then $A$ is a commutator. This is achieved by extending the methods
of Laffey and Reams and in particular removing the need for Dirichlet's
theorem. Another of our main results is a certain (non-unique) normal
form for similarity classes of matrices over PIDs, itself a generalisation
of a result proved in \cite{Laffey-Reams} over $\Z$. The normal
form, while interesting in its own right and potentially for other
applications, is also a key ingredient in the proof of the main result
on commutators.

We now describe the contents of the paper in more detail. In Section~\ref{sec:Regular-elements}
we define regular elements in $\M_{n}(R)$ for an arbitrary ring $R$
and state some of their basic properties. Regular elements play a
central role in the problem of writing matrices as commutators because
of the criterion of Laffey and Reams, treated in Section~\ref{sec:LF-criterion}.
The criterion says that if $R$ is a PID and $A,X\in\M_{n}(R)$ with
$X$ regular mod every maximal ideal of $R$, then a necessary and
sufficient condition for $A$ to be a commutator is that $\Tr(X^{r}A)=0$
for $r=0,1,\dots,n-1$. This was proved in \cite{Laffey-Reams} for
$R=\Z$, but the proof goes through for any PID with only a minor
modification.

In Section~\ref{sec:Comm-fields} we apply the Laffey-Reams criterion
for fields to give a short proof of the theorem of Albert and Muckenhoupt
mentioned above. We actually prove a stronger and apparently new result,
namely that in the commutator one of the matrices may be taken to
be regular (see Proposition~\ref{sec:Comm-fields}).

Section~\ref{sec:Similarity} is concerned with similarity of matrices
over PIDs, that is, matrices up to conjugation by invertible elements.
Our first main result is Theorem~\ref{thm:LF-normalform} stating
that every non-scalar element in $\M_{n}(R)$ is similar to one in
a special form. This result was established by Laffey and Reams over
$\Z$. However, a crucial step in their proof uses the fact that $2$
is a prime element in $\Z$, and the analogue of this does not hold
in an arbitrary PID. To overcome this, our proof involves an argument
based on the surjectivity of the map $\SL_{n}(R)\rightarrow\SL_{n}(R/I)$
for an ideal $I$, which in a certain sense lets us avoid any finite
set of primes, in particular those of index $2$ in $R$ (see Lemma~\ref{lem:b12-avoidsprimes}).
This argument is evident especially in the proof of Proposition~\ref{prop:3x3-normalform}.
Apart from this, our proof uses the methods of \cite{Laffey-Reams},
although we give a different argument, avoiding case by case considerations,
and have made Lemma~\ref{lem:row-column} explicit.

Our second main result is Theorem~\ref{thm:Main} whose proof occupies
Section~\ref{sec:Proof-Main}, and follows the lines of \cite[Section~4]{Laffey-Reams}.
There are two new key ideas in our proof. First, there is again an
argument which at a certain step allows us to avoid finitely many
primes, including those of index $2$ in $R$. This step in the proof
is the choice of $q$ and uses a special case of Lemma~\ref{lem:GCD}\,\ref{enu:GCD-lemma abx}.
Secondly, we apply Lemma~\ref{lem:Centr-product} to obtain a set
of generators of the centraliser of a certain matrix modulo a product
of distinct primes; see (\ref{eq:Centr-span-severalprimes}). It is
this set of generators together with our choice of $q$ and an appropriate
choice of $t$ in (\ref{eq:at+y}) which allows us to avoid Dirichlet's
theorem. It is interesting to note that the proofs of our main results,
Theorems~\ref{thm:LF-normalform} and \ref{thm:Main}, despite being
rather different, both involve the technique of avoiding finitely
many primes, in particular those of index $2$ in $R$. Our proof
of Theorem~\ref{thm:Main} also simplifies parts of the proof of
Laffey and Reams over $\Z$ since we avoid some of the case by case
considerations present in the latter. By a theorem of Hungerford,
Theorem~\ref{thm:Main}, once established, easily extends to any
principal ideal ring (not necessarily an integral domain); see Corollary~\ref{cor:Coroll-Main}.

The final Section~\ref{sec:Further-directions} discusses the possibility
of generalising Theorem~\ref{thm:Main} to other classes of rings
such as Dedekind domains, and mentions some known counter-examples.

We end this introduction by mentioning some recent work on matrix
commutators. In \cite{Mesyan} Mesyan proves that if $R$ is a ring
(not necessarily commutative) and $A\in\M_{n}(R)$ has trace zero,
then $A$ is a sum of two commutators. This result was proved for
commutative rings in earlier unpublished work of Rosset. In \cite{Lam-Khurana-Gen-comm}
Khurana and Lam study {}``generalised commutators'', that is, elements
of the form $XYZ-ZYX$, where $X,Y,Z\in\M_{n}(R)$. They establish
in particular that if $R$ is a PID, then every element in $\M_{n}(R)$,
$n\geq2$, is a generalised commutator. Although these results may
seem closely related to the commutator problem studied in the present
paper, the proofs are in fact very different.

\subsection*{Notation and terminology}

We use $\N$ to denote the natural numbers $\{1,2,\dots\}$. Throughout
the paper a ring will always mean a commutative ring with identity.
In Sections~\ref{sec:LF-criterion}-\ref{sec:Proof-Main} $R$ will
be a PID, unless stated otherwise.

Let $R$ be a ring. We denote the set of maximal ideals of $R$ by
$\Specm R$ and the ring of $n\times n$ matrices over $R$ by $\M_{n}(R)$.
For $A,B\in\M_{n}(R)$ we call $[A,B]=AB-BA$ the \emph{commutator}
of $A$ and $B$. Let $A\in\M_{n}(R)$. A matrix $B\in\M_{n}(R)$
is said to be \emph{similar} to $A$ if there exists a $g\in\GL_{n}(R)$
such that $gAg^{-1}=B$. The transpose of $A$ is denoted by $A^{T}$
and the trace of $A$ by $\Tr(A)$. We write $C_{\M_{n}(R)}(A)$ for
the centraliser of $A$ in $\M_{n}(R)$, that is,
\[
C_{\M_{n}(R)}(A)=\{B\in\M_{n}(R)\mid[A,B]=0\}.
\]
Let $f(x)=a_{0}+a_{1}x+\dots+x^{n}\in R[x]$ be the characteristic
polynomial of $A$. We will refer to the \emph{companion matrix} associated
to $A$ (or to $f$) as the matrix $C\in\M_{n}(R)$ such that
\[
C=(c_{ij})=\begin{cases}
c_{i,i+1}=1 & \text{for }1\leq i\leq n-1,\\
c_{ni}=-a_{i-1} & \text{for }1\leq i\leq n,\\
c_{ij}=0 & \text{otherwise}.
\end{cases}
\]
The identity matrix in $\M_{n}(R)$ is denoted by $1$ or sometimes
$1_{n}$. For $u,v\in\N$ we write $E_{uv}$ for the matrix units,
that is, $E_{uv}=(e_{ij})$ with $e_{uv}=1$ and $e_{ij}=0$ otherwise.
The size of the matrices $E_{uv}$ is suppressed in the notation and
will be determined by the context.

\section{\label{sec:Regular-elements}Regular elements}

Let $\bfG$ be a reductive algebraic group over a field $K$ with
algebraic closure $\overline{K}$. An element $x\in G=\bfG(\overline{K})$
is called \emph{regular} if $\dim C_{G}(x)$ is minimal, and it is
known that this minimal dimension equals the rank $\rk G$ (see \cite{Steinberg-regular}
and \cite[Section~14]{dignemichel}). Similarly, if $\mfg$ is the
Lie algebra of $\bfG$ an element $X\in\mfg(\overline{K})$ is called
\emph{regular} if $\dim C_{G}(X)=\rk G$, where $G$ acts on $\mfg$
via the adjoint action. In the case $\bfG=\GL_{n}$ there are several
equivalent characterisations of regular elements in $\mfg(K)=\M_{n}(K)$.
More precisely, the following is well-known:
\begin{prop}
\label{prop:Reg-fields}Let $K$ be a field and $X\in\M_{n}(K)$.
Then the following is equivalent
\begin{enumerate}
\item \label{enu:reg-fields 1}$X$ is regular,
\item \label{enu:reg-fields 5-1}There exists a vector $v\in K^{n}$ such
that $\{v,Xv,\dots,X^{n-1}v\}$ is a basis for $K^{n}$ over $K$,
\item \label{enu:reg-fields 5}The set $\{1,X,\dots,X^{n-1}\}$ is linearly
independent over $K$,
\item \label{enu:reg-fields 1-1}$X$ is similar to its companion matrix
$C$ as well as to $C^{T}$,
\item \label{enu:reg-fields 2}$C_{\M_{n}(K)}(X)=K[X]$.
\end{enumerate}
\end{prop}
Regular elements of $\M_{n}(K)$ are sometimes called \emph{non-derogatory
}or\emph{ cyclic.} For matrices over arbitrary rings we make the following
definition.
\begin{defn}
Let $R$ be a ring. A matrix $X\in\M_{n}(R)$ is called \emph{regular}
if there exists a vector $v\in R^{n}$ such that $\{v,Xv,\dots,X^{n-1}v\}$
is a basis for $R^{n}$ over $R$.\end{defn}
\begin{prop}
\label{prop:Reg-rings}Let $R$ be a ring and $X\in\M_{n}(R)$. Then
the following is equivalent
\begin{enumerate}
\item \label{enu:reg-rings 1}$X$ is regular,
\item \label{enu:reg-rings 2-1}$X$ is similar to its companion matrix
$C$ as well as to $C^{T}$,
\item \label{enu:reg-rings 3}$C_{\M_{n}(R)}(X)=R[X]$.
\end{enumerate}
\end{prop}
The proof of Proposition \ref{prop:Reg-rings} is the same as in the
classical case of matrices over fields. In the following we will use
the properties of regular elements expressed in Propositions~\ref{prop:Reg-fields}
and \ref{prop:Reg-rings} without explicit reference.

If $\phi:R\to S$ is a homomorphism of rings we also use $\phi$ to
denote the induced homomorphism $\M_{n}(R)\rightarrow\M_{n}(S)$.
\begin{lem}
\label{lem:Reg-extnscalars}Let $\phi:R\to S$ be a homomorphism of
rings. If $X\in\M_{n}(R)$ is regular, then $\phi(X)$ is regular. \end{lem}
\begin{proof}
Suppose that $X$ is regular. By definition there exists a vector
$v\in R^{n}$ such that $\{v,Xv,\dots,X^{n-1}v\}$ is an $R$-basis
for $R^{n}$. Then $\{v\otimes1,Xv\otimes1,\dots,X^{n-1}v\otimes1\}$
is an $S$-basis for $R^{n}\otimes_{R}S$ (cf.~\cite[XVI, Proposition~2.3]{Lang-Algebra}).
Let $\phi(v)\in S^{n}$ be the image of $v$ under component-wise
application of $\phi$. Under the isomorphism $R^{n}\otimes_{R}S\rightarrow S^{n}$,
the elements $X^{i}v\otimes1$ are sent to $\phi(X)^{i}\phi(v)$,
so $\{\phi(v),\phi(X)\phi(v),\dots,\phi(X)^{n-1}\phi(v)\}$ is a basis
for $S^{n}$. Thus $\phi(X)$ is regular.
\end{proof}
Let $R$ be a ring and $X\in\M_{n}(R)$. If $\mfp$ is an ideal of
$R$ we use $X_{\mfp}$ to denote the image of $X$ under the canonical
map $\pi:\M_{n}(R)\to\M_{n}(R/\mfp$), that is, $X_{\mfp}=\pi(X)$.
For a general ring $R$ an element in $\M_{n}(R)$ which is regular
modulo every maximal ideal may not be regular. However, if $R$ is
a local ring, the situation is favourable:
\begin{lem}
\label{lem:Reg-locring}Assume that $R$ is a local ring with maximal
ideal $\mfm$. Then $X\in\M_{n}(R)$ is regular if and only if $X_{\mfm}\in\M_{n}(R/\mfm)$
is regular.\end{lem}
\begin{proof}
If $X$ is regular, then $X_{\mfm}$ is regular by Lemma~\ref{lem:Reg-extnscalars}.
Conversely, suppose that $X_{\mfm}$ is regular and choose $v\in(R/\mfm)^{n}$
such that $(R/\mfm)^{n}=(R/\mfm)[X_{\mfm}]v$. Let $\hat{v}\in R^{n}$
be a lift of $v$. Then $R^{n}=R[X]\hat{v}+\mfm M$ for some submodule
$M$ of $R^{n}$, and Nakayama's lemma yields $R^{n}=R[X]\hat{v}$
, so $X$ is regular.\end{proof}
\begin{prop}
\label{prop:Reg-mod-m}Let $R$ be an integral domain with field of
fractions $F$, and let $X\in\M_{n}(R)$. If $X_{\mfm}$ is regular
for some maximal ideal $\mfm$ of $R$, then $X$ is regular as an
element of $\M_{n}(F)$. \end{prop}
\begin{proof}
Suppose that $X_{\mfm}$ is regular for some maximal ideal $\mfm$
of $R$. Let $R_{\mfm}$ be the localisation of $R$ at $\mfm$, and
let $j:R\to R_{\mfm}$ be the canonical homomorphism. Since the diagram
\[
\xymatrix{R\ar[d]\ar[r]^{j} & R_{\mfm}\ar[d]\\
R/\mfm\ar[r]\sp-{\cong} & R_{\mfm}/\mfm
}
\]
commutes, Lemma~\ref{lem:Reg-locring} implies that $j(X)$ is regular.
If $\sum_{i=0}^{n-1}r_{i}X^{i}=0$ for some $r_{i}\in R$, then $\sum_{i=0}^{n-1}j(r_{i})j(X)^{i}=0$.
But since $j(X)$ is regular, we must have $j(r_{i})=0$ for all $i=0,\dots,n-1$.
Since $R$ is an integral domain $j$ is injective, so $r_{i}=0$
for $i=0,\dots,n-1$. Now, if $\sum_{i=0}^{n-1}s_{i}X^{i}=0$ for
some $s_{i}\in F$, then clearing denominators shows that $s_{i}=0$
for all $i=0,\dots,n-1$. Hence, by Proposition~\ref{prop:Reg-fields}~\ref{enu:reg-fields 5}
the matrix $X$ is regular as an element of $\M_{n}(F)$.
\end{proof}
The following result has appeared in \cite[Proposition~6]{Vaserstein-Wheland}.
\begin{lem}
\label{lem:reg-triang}Let $R$ be an arbitrary ring and $A=(a_{ij})\in\M_{n}(R)$
a matrix such that $a_{i,i+1}=1$ for all $1\leq i\leq n$ and $a_{ij}=0$
for all $j\geq i+2$. Then $A$ is regular. \end{lem}
\begin{proof}
Let $\{e_{1}=(1,0,\dots,0)^{T},e_{2}=(0,1,0,\dots,0)^{T},\dots,e_{n}=(0,\dots,0,1)^{T}\}$
be the standard basis for $R^{n}$. Then the matrix
\[
B=(e_{1},Ae_{1},\dots,A^{n-1}e_{1})
\]
is upper triangular with $1$s on the diagonal, so $B\in\SL_{n}(R)$.
Now for $1\leq i\leq n-1$ we have
\[
B^{-1}ABe_{i}=B^{-1}A^{i}e_{1}=e_{i+1}
\]
(since $Be_{i+1}=A^{i}e_{1}$ ). Thus $B^{-1}AB$ is a companion matrix,
and so $A$ is regular.
\end{proof}

\section{\label{sec:LF-criterion}\foreignlanguage{british}{The criterion
of Laffey and Reams }}

\selectlanguage{british}%
Throughout this section $R$ is a PID and $F$ its field of fractions.
In Theorem~\ref{prop:Criterion} we give a criterion for a matrix
in $\M_{n}(R)$ to be a commutator discovered by Laffey and Reams
\cite[Section~3]{Laffey-Reams}. This criterion plays an important
role in our proof of the main theorem. Laffey and Reams proved the
criterion for matrices over fields and over $\Z$, and we only need
minor modifications of their proofs, together with Proposition~\ref{prop:Reg-mod-m},
to prove it over arbitrary PIDs.

The following result is from \cite[Section~3]{Laffey-Reams}. We reproduce
the proof here for completeness.
\begin{prop}
\label{prop:LF-criterion-fields}Let $K$ be a field and $X\in\M_{n}(K)$
be regular. Let $A\in\M_{n}(K)$. Then $A=[X,Y]$ for some $Y\in\M_{n}(K)$
if and only if \foreignlanguage{english}{$\Tr(X^{r}A)=0$ for all
$r=0,\dots,n-1$.}\end{prop}
\selectlanguage{english}%
\begin{proof}
Since $\{1,X,\dots,X^{n-1}\}$ is linearly independent over $K$ the
subspace
\[
V=\{A\in\M_{n}(K)\mid\Tr(X^{r}A)=0\text{ for }0,1,\dots,n-1\}
\]
has dimension $n^{2}-n$. The kernel of the linear map $\M_{n}(R)\rightarrow\M_{n}(R)$,
$Y\mapsto[X,Y]$ is equal to the centraliser $C_{\M_{n}(K)}(X)$,
which has dimension $n$ since $X$ is regular. Thus the image $[X,\M_{n}(K)]$
of the map $Y\mapsto[X,Y]$ has dimension $n^{2}-n$. But if $A\in[X,\M_{n}(K)]$
there exists a $Y\in\M_{n}(K)$ such that for every $r=0,1,\dots,n-1$
we have
\[
\Tr(X^{r}A)=\Tr(X^{r}(XY-YX))=\Tr(X^{r+1}Y)-\Tr(X^{r}YX)=0.
\]
Thus $A\in V$ and so $[X,\M_{n}(K)]\subseteq V$. Since $\dim V=\dim[X,\M_{n}(K)]$
we conclude that $V=[X,\M_{n}(K)]$. \end{proof}
\selectlanguage{british}%
\begin{prop}
\label{prop:LF-XYM}Let $X\in\M_{n}(R)$ be such that $X_{\mfp}$
is regular for every maximal ideal $\mfp$ in $R$. Suppose that $M\in\M_{n}(F)$
is such that $[X,M]\in\M_{n}(R)$. Then there exists an $Y\in\M_{n}(R)$
such that $[X,M]=[X,Y]$.\end{prop}
\selectlanguage{english}%
\begin{proof}
There exists an element $m\in R$ such that $mY\in\M_{n}(R)$, and
we have $[X,mY]=m[X,Y]$. Assume that $d\in R$ is chosen so that
it has the minimal number of irreducible factors with respect to the
property that $[X,C]=d[X,Y]$ for some $C\in\M_{n}(R)$. If $d$ is
a unit we are done, so assume that $p$ is an irreducible factor of
$d$. Then $[X,C]\in p\M_{n}(R)$, so $X_{(p)}$ commutes with $C_{(p)}$.
But since $X_{(p)}$ is regular, we have $C_{(p)}=f(X_{(p)})$, for
some polynomial $f(T)\in R[T]$. Hence $C-f(X)=pD$ for some $D\in\M_{n}(R)$.
But this implies that $[X,C]=[X,pD]=p[X,D]$ and thus $(dp^{-1})[X,Y]=[X,D]$,
giving a contradiction to our choice of $d$. Hence $d$ is a unit
and so $[X,Y]=[X,M]$ with $M=d^{-1}C\in\M_{n}(R)$.\end{proof}
\begin{prop}
\label{prop:Criterion}Let $A\in\M_{n}(R)$ and let $X\in\M_{n}(R)$
be such that $X_{\mfp}$ is regular for every maximal ideal $\mfp$
in $R$. Then $A=[X,Y]$ for some $Y\in\M_{n}(R)$ if and only if
$\Tr(X^{r}A)=0$ for $r=0,\dots,n-1$.\end{prop}
\begin{proof}
Clearly the condition $\Tr(X^{r}A)=0$ for all $r\geq0$ is necessary
for $A$ to be of the form $[X,Y]$ with $Y\in\M_{n}(R)$. Conversely,
suppose that $\Tr(X^{r}A)=0$ for $r=0,1,\dots,n-1$. By \foreignlanguage{british}{Proposition~\ref{prop:Reg-mod-m}}
$X$ is regular \foreignlanguage{british}{as an element in $\M_{n}(F)$
so Proposition~\ref{prop:LF-criterion-fields} implies that $A=[X,M]$
for some $M\in\M_{n}(F)$. But now the result follows from Proposition~\ref{prop:LF-XYM}.}
\end{proof}

\section{\label{sec:Comm-fields}Commutators over fields}

Let $K$ be a field. Using the criterion of Laffey and Reams over
fields (Proposition~\ref{prop:LF-criterion-fields}) we give a swift
proof of the theorem of Albert and Muckenhoupt \cite{Albert-Muckenhoupt}
that every matrix with trace zero in $\M_{n}(K)$ is a commutator.

Note that if $R$ is any ring and $A,X,Y\in\M_{n}(R)$ are such that
$A=[X,Y]$, then for every $g\in\GL_{n}(R)$ we have $gAg^{-1}=[gXg^{-1},gYg^{-1}]$.
Thus $A$ is a commutator if and only if any matrix similar to $A$
is.

Let $n\in\N$ with $n\geq2$ and $k=\lfloor n/2\rfloor$. The following
matrices were considered by Laffey and Reams \cite[Section~4]{Laffey-Reams}
who also established the properties stated below.
\[
P_{n}=(p_{ij})=\begin{cases}
p_{ii}=1 & \text{for }i=2,4,\dots,2k,\\
p_{i,i-2}=1 & \text{for }i=3,4,\dots n,\\
p_{ij}=0 & \text{otherwise}.
\end{cases}
\]
Depending on the context we will consider $P_{n}$ as an element of
$\M_{n}(R)$ where $R$ is a ring. For any $m\in\N$ and $a\in R$
we will use $J_{m}(a)$ to denote the $m\times m$ Jordan block with
eigenvalue $a$ and $1$s on the subdiagonal. Over any $R$ the matrix
$P_{n}$ is similar to $J_{k}(1)\oplus J_{n-k}(0)$ (cf.~\cite[p.~681]{Laffey-Reams}),
and thus it is regular by Lemma~\ref{lem:reg-triang}.

For any $A=(a_{ij})\in\M_{n}(R)$ let $c(A)=\sum_{i=1}^{k}a_{2i,2i}$
and $d(A)=\sum_{i=1}^{n-1}a_{i,i+1}$. Suppose now that $R$ is a
PID and that $a_{ij}=0$ for $j\geq i+2$. Observe that for any $r\in\N$,
$P_{n}^{r}$ has the same diagonal as $P_{n}$ and the $(i,j)$ entry
of $P_{n}^{r}$ is $0$ if $i\neq j$ and $i<j+2$. Thus
\begin{equation}
\Tr(P_{n}^{r}A)=c(A),\,\text{ for }r\in\N.\label{eq:Tr_c(A)}
\end{equation}

\begin{prop}
\label{prop:Main-fields}Let $K$ be a field and let $A\in\M_{n}(K)$
be a matrix with trace zero. Then $A=[X,Y]$ for some $X,Y\in\M_{n}(K)$,
where $X$ is regular. More precisely, if $A$ is non-scalar $X$
can be chosen to be conjugate to $P_{n}$, while if $A$ is scalar
we can take $X=J_{n}(0)$.\end{prop}
\begin{proof}
Assume first that $A$ is non-scalar. It then follows from the rational
normal form that $A$ is similar to a matrix $B=(b_{ij})$ with $b_{12}=1$
and $b_{ij}=0$ for $j\geq i+2$, so we have $A=gBg^{-1}$ for some
$g\in\GL_{n}(K)$. Define $z\in\SL_{n}(K)$ as
\[
z=1+c(B)E_{21}.
\]
Then the $(i,j)$ entry of $z^{-1}Bz$ is $0$ for $j\geq i+2$ and
$c(z^{-1}Bz)=0$, so by (\ref{eq:Tr_c(A)}) we have $\Tr(P_{n}^{r}z^{-1}Bz)=0$
for $r=0,\dots,n-1$. By Proposition~\ref{prop:LF-criterion-fields}
it follows that $B=[zP_{n}z^{-1},Y]$ for some $Y\in\M_{n}(K)$, and
thus $A=[gzP_{n}(gz)^{-1},gYg^{-1}]$.

Assume on the other hand that $A$ is a scalar. Then $\Tr(J_{n}(0)^{r}A)=0$
for $r=0,\dots,n-1$, and Proposition~\ref{prop:LF-criterion-fields}
implies that $A=[J_{n}(0),Y]$, for some $Y\in\M_{n}(K)$.
\end{proof}

\section{\label{sec:Similarity}Matrix similarity over a PID}

In this section we extend the results of \cite[Section~2]{Laffey-Reams}
on similarity of matrices over $\Z$ to matrices over an arbitrary
PID $R$.
\begin{lem}
\label{lem:b12-avoidsprimes}Let $A\in\M_{n}(R)$ be non-scalar, and
let $S$ be a finite set of maximal ideals of $R$ such that $A_{\mfp}\in\M_{n}(R/\mfp)$
is non-scalar for every $\mfp\in S$. Then $A$ is similar to a matrix
$B=(b_{ij})\in\M_{n}(R)$ such that $b_{12}\notin\mfp$ for all $\mfp\in S$.\end{lem}
\begin{proof}
It is well known that for any PID $R$ and any non-zero ideal $\mfa$
of $R$ the natural map
\begin{equation}
\SL_{n}(R)\longrightarrow\SL_{n}(R/\mfa)\label{eq:strongapprox}
\end{equation}
is surjective. This follows for example from the fact that $R/\mfa$
is the product of local rings and that over local rings $\SL_{n}$
is generated by elementary matrices (see~\cite[2.2.2~and~2.2.6]{Rosenberg_K-theory}).
Moreover, if we take $\mfa=\prod_{\mfp\in S}\mfp$ the Chinese remainder
theorem implies that we have an isomorphism
\begin{equation}
\SL_{n}(R/\mfa)\longiso\prod_{\mfp\in S}\SL_{n}(R/\mfp).\label{eq:Chineserem}
\end{equation}
Let $\mfp\in S$. Since $A_{\mfp}$ is non-scalar and $R/\mfp$ is
a field the rational canonical form for matrices in $\M_{n}(R/\mfp)$
implies that there exists a $g_{\mfp}\in\GL_{n}(R/\mfp)$ such that
$g_{\mfp}A_{\mfp}g_{\mfp}^{-1}$ is a matrix whose $(1,2)$ entry
is non-zero. Since $\GL_{n}(R/\mfp)=T(R/\mfp)\SL_{n}(R/\mfp)$, where
$T(R/\mfp)$ is the diagonal subgroup of $\GL_{n}(R/\mfp)$, we may
take $g_{\mfp}$ to be in $\SL_{n}(R/\mfp)$. Suppose that $g_{\mfp}$
is chosen in this way for every $\mfp\in S$. By the surjectivity
of the maps (\ref{eq:strongapprox}) and (\ref{eq:Chineserem}), there
exists a $g\in\SL_{n}(R)$ such that the image of $g$ in $\SL_{n}(R/\mfp)$
is $g_{\mfp}$ for all $\mfp\in S$. Let $B=(b_{ij})=gAg^{-1}$. Then
$B$ is a matrix such that $b_{12}$ is non-zero modulo every $\mfp\in S$.
\end{proof}
The following lemma will be used repeatedly in the proof of Proposition~\ref{prop:3x3-normalform}
and Theorem~\ref{thm:LF-normalform}. It can informally be described
as saying that if the off-diagonal entries in a row (column) of a
matrix $A\in\M_{n}(R)$ with $n\geq3$ have a greatest common divisor
$d$, then $A$ is similar to a matrix in which the corresponding
row (column) has off-diagonal entries $d,0,\dots,0$.
\begin{lem}
\label{lem:row-column}Let $A=(a_{ij})\in\M_{n}(R)$, $n\geq3$. Let
$1\leq u\leq n$ and $1\leq v\leq n$ be fixed. Let $r\in R$ be a
generator of the ideal $(a_{uj}\mid1\leq j\leq n,\, u\neq j)$, and
let $c\in R$ be a generator of the ideal $(a_{iv}\mid1\leq i\leq n,\, i\neq v)$.
Then $A$ is similar to a matrix $B=(b_{ij})$ such that if $u=1$
we have $b_{u2}=r$ and $b_{uj}=0$ for all $3\leq j\leq n$, and
if $u\geq2$ we have $b_{u1}=r$ and $b_{uj}=0$ for all $1\leq j\leq n$
such that $j\notin\{1,u\}$. Moreover, $A$ is similar to a matrix
$C=(c_{ij})$ such that if $v=1$ we have $c_{2v}=r$ and $c_{iv}=0$
for all $3\leq i\leq n$, and if $v\geq2$ we have $c_{1v}=c$ and
$c_{iv}=0$ for all $1\leq i\leq n$ such that $i\notin\{1,v\}$.\end{lem}
\begin{proof}
The proof follows the lines of \cite[Ch.~III, Section~2]{Newman}.
For $1\leq i<j\leq n$ and $\left(\begin{smallmatrix}x & y\\
z & w
\end{smallmatrix}\right)\in\SL_{2}(R)$, let
\begin{align*}
M_{ij} & =M_{ij}(x,y,z,w)\\
 & =1_{n}+(x-1)E_{ii}+yE_{ij}+zE_{ji}+(w-1)E_{jj}\in\SL_{n}(R).
\end{align*}
Note that $M_{ij}^{-1}=M_{ij}(w,-y,-z,x)$. Let $3\leq j\leq n$.
Direct computation shows that the first row in $B_{1}\coloneqq M_{2j}^{-1}AM_{2j}$
is
\begin{align*}
(a_{11},a_{12}x+a_{13}z,a_{12}y+a_{13}w,a_{14},\dots,a_{1n}) & \quad\text{if }j=3,\\
(a_{11},a_{12}x+a_{1j}z,a_{13},\dots,a_{1,j-1},a_{12}y+a_{1j}w,a_{1,j+1},\dots,a_{1n}) & \quad\text{if }j>3.
\end{align*}
Now let $3\leq j\leq n$ be the smallest integer such that $a_{1j}\neq0$
(if no such $j$ exists the assertion of the lemma holds trivially
for $A$ and $u=1$). Let $d\in R$ be a generator of $(a_{12},a_{1j})$
and set
\[
y=a_{1j}d^{-1},\quad w=-a_{12}d^{-1}.
\]
Then $(y,w)=(1)$ and hence $x,z\in R$ may be determined so that
$xw-yz=1$. Thus $a_{12}x+a_{1j}z=-d$. With these values of $x,y,z,w$
all the entries of $A_{1}$ in positions $(1,3),\dots,(1,j)$ are
zero, and the $(1,2)$ entry generates the ideal ($a_{12},a_{1j})$.
Repeating the process, let $j<k\leq n$ be the smallest integer such
that $a_{1k}\neq0$. Then $B_{2}\coloneqq M_{2k}^{-1}B_{1}M_{2k}$
has all its entries $(1,3),\dots,(1,k)$ zero and its $(1,2)$ entry
generates the ideal $(a_{12},a_{1j},a_{1k})$. Proceeding in this
way, we obtain a matrix $B=(b_{ij})$ similar to $A$ such that $b_{12}$
is a generator of $(a_{1j}\mid2\leq j\leq n)$ and $b_{1j}=0$ for
$3\leq j\leq n$ (the generator $b_{12}$ can be replaced by any other
generator of $(a_{1j}\mid2\leq j\leq n)$ by a diagonal similarity
transformation of $B$). This shows the existence of $B$ for $u=1$.
For $u\geq2$, observe that if we let $W_{u}=(w_{ij}^{(u)})\in\GL_{n}(R)$
be any permutation matrix such that $w_{1u}^{(u)}=w_{u1}=1$, then
\[
A'=(a_{ij}')=W_{u}AW_{u}^{-1}
\]
is a matrix such that $a_{11}'=a_{uu}$ and $\{a_{1j}'\mid2\leq j\leq n\}=\{a_{uj}\mid1\leq j\leq n,u\neq j\}$.
Informally, the off-diagonal entries in the $u$-th row of $A$ are
the same as the off-diagonal entries in the first row of $A'$, up
to a permutation. Thus the existence of $B$ for $u\geq2$ follows
from the argument for $u=1$ above.

For the existence of $C$ for $v=1$, let $3\leq i\leq n$ and $C_{1}\coloneqq M_{2i}^{-1}AM_{2i}$.
Direct computation shows that the first column in $C_{1}$ is
\begin{align*}
(a_{11},a_{21}x+a_{31}y,a_{21}z+a_{31}w,a_{41},\dots,a_{n1})^{T} & \quad\text{if }i=3,\\
(a_{11},a_{21}x+a_{i1}y,a_{31},\dots,a_{i-1,1},a_{21}z+a_{i1}w,a_{i+1,1},\dots,a_{n1})^{T} & \quad\text{if }i>3.
\end{align*}
Now let $3\leq i\leq n$ be the smallest integer such that $a_{i1}\neq0$
(if no such $i$ exists the assertion of the lemma holds trivially
for $A$ and $v=1$). Let $e\in R$ be a generator of $(a_{21},a_{i1})$
and set
\[
z=a_{i1}e^{-1},\quad w=-a_{21}d^{-1}.
\]
Then $(z,w)=(1)$ and hence $x,y\in R$ may be determined so that
$xw-yz=1$. Thus $a_{21}x+a_{i1}y=-e$. With these values of $x,y,z,w$
all the entries of $C_{1}$ in positions $(3,1),\dots,(i,1)$ are
zero, and the $(2,1)$ entry generates the ideal ($a_{21},a_{i1})$.
Repeating the process in analogy with the above argument, we obtain
a matrix $C$ satisfying the assertion of the lemma for $v=1$. For
$v\geq2$ we may use the matrix $W_{v}$ as above to reduce to the
case where $v=1$.\end{proof}
\begin{prop}
\label{prop:3x3-normalform}Let $A\in\M_{3}(R)$ be non-scalar. Then
$A$ is similar to a matrix $B=(b_{ij})\in\M_{3}(R)$ such that $b_{12}\mid b_{ij}$
for all $i\neq j$ and $b_{12}\mid(b_{ii}-b_{jj})$ for all $1\leq i,j\leq3$.\end{prop}
\begin{proof}
Write $A=aI+bA'$, where $a,b\in R$, $b\neq0$ and where, if $A'=(a_{ij}')$,
we have $(a_{ii}'-a_{jj}',a_{ij}'\mid i\neq j,1\leq i,j\leq3)=(1)$.
Note that $A_{\mfp}'$ is non-scalar for every maximal ideal $\mfp$
of $R$ and that the proposition will follow for $A$ if we can show
it for $A'$, that is, if we can show that $A'$ is similar to a matrix
whose $(1,2)$ entry is a unit. Without loss of generality we may
therefore assume that $A=A'$ so that $A$ satisfies
\[
(a_{ii}-a_{jj},a_{ij}\mid i\neq j,1\leq i,j\leq3)=(1).
\]
Note that any matrix similar to $A$ will also satisfy this. Let
\[
S\coloneqq\{\mfp\in\Specm R\mid|R/\mfp|=2\}.
\]
Note that $S$ is a finite set since in any PID (or any Dedekind domain)
there are only finitely many maximal ideals of any given finite index.
Since $A_{\mfp}$ is not scalar for any maximal ideal $\mfp$ of $R$,
Lemma~\ref{lem:b12-avoidsprimes} implies that $A$ is similar to
a matrix $B=(b_{ij})$ such that $b_{12}\notin\mfp$ for all $\mfp\in S$.
Among all such matrices choose one for which the number of distinct
primes which divide $b_{12}$ is least possible, and subject to this,
for which the number of not necessarily distinct prime factors is
minimal. By Lemma~\ref{lem:row-column} applied to the first row
in $B$, we see that there exists a matrix $B'$ similar to $B$ whose
$(1,3)$ entry is zero and whose $(1,2)$ entry, being equal to a
generator of $(b_{12},b_{13})$, has no more distinct prime factors
than $b_{12}$. Hence we may assume that $B$ has been replaced by
$B'$ so that $b_{13}=0$. We thus have the following condition on
$B$:

\MyQuote{The matrix $B=(b_{ij})$ is similar to $A$, $b_{12}\notin\mfp$ for all $\mfp\in S$, $b_{13}=0$, the entry $b_{12}$ has the smallest number of distinct prime factors among all the matrices similar to $A$ and among all matrices with these properties $B$ is such that $b_{12}$ has the minimal number of not necessarily distinct prime factors.}Note
first that by Lemma~\ref{lem:row-column} applied to the second column
in $B$, there exists a matrix similar to $B$ whose $(1,2)$ entry
is a generator of $(b_{12},b_{32})$. Thus, by $(*)$ we must have
$b_{12}\mid b_{32}$, so $b_{32}=b_{12}a$ for some $a\in R$. Let
\[
B_{1}=(b_{ij}^{(1)})=(1-E_{31}a)B(1-E_{31}a)^{-1}.
\]
Then $b_{12}^{(1)}=b_{12}$ and $b{}_{13}^{(1)}=b_{32}^{(1)}=0$ so
that
\[
B_{1}=\begin{pmatrix}b_{11}^{(1)} & b_{12} & 0\\
b_{21}^{(1)} & b_{22}^{(1)} & b_{23}^{(1)}\\
b_{31}^{(1)} & 0 & b_{33}^{(1)}
\end{pmatrix}.
\]
In particular, $B'$ satisfies $(*)$.
\begin{claim}
\label{Claim I}The entry $b_{12}$ divides both $b_{33}^{(1)}-b_{11}^{(1)}$
and $b_{31}^{(1)}$.
\end{claim}
Let $y\in R$. The first row of the matrix $(1+E_{13}y)B_{1}(1+E_{13}y)^{-1}$
is
\[
(b_{11}^{(1)}+yb_{31}^{(1)},\, b_{12},\, y(b_{33}^{(1)}-b_{11}^{(1)}-yb_{31}^{(1)})).
\]
Thus, by $(*)$ and Lemma~\ref{lem:row-column} applied to the first
row in $(1+E_{13}y)B_{1}(1+E_{13}y)^{-1}$ we conclude that $b_{12}$
divides $y(b_{33}^{(1)}-b_{11}^{(1)}-yb_{31}^{(1)})$ for any $y\in R$.
Let
\[
(b_{12})=\mfp_{1}^{e_{1}}\cdots\mfp_{\nu}^{e_{\nu}}
\]
be the factorisation of $(b_{12})$, where $\nu\in\N$, $e_{i}\in\N$
and the ideals $\mfp_{i}\in\Specm R$ are distinct for $1\leq i\leq\nu$.
By $(*)$ and the definition of $S$ we know that $|R/\mfp_{i}|\geq3$
for any $1\leq i\leq\nu$. Hence there exist elements $y_{i},y_{i}'\in R/\mfp_{i}$
such that
\begin{equation}
y_{i}\neq0,\quad y_{i}'\neq0,\quad y_{i}\neq y_{i}',\quad\text{for }i=1,\dots,\nu.\label{eq:yi-yi'}
\end{equation}
By the Chinese remainder theorem we have
\[
R/(b_{12})\cong\prod_{i=1}^{\nu}R/\mfp_{i}^{e_{i}}.
\]
Let $\lambda=(y_{1},\dots,y_{\nu}),\lambda'=(y_{1}',\dots,y_{\nu}')\in\prod_{i=1}^{\nu}R/\mfp_{i}^{e_{i}}$.
Then $\lambda$ and $\lambda'$ can be considered as elements in $R/(b_{12})$
and because of (\ref{eq:yi-yi'}) each of $\lambda,\lambda'$ and
$\lambda-\lambda'$ is a unit in $R/(b_{12})$. In particular, each
of $\lambda,\lambda'$ and $\lambda-\lambda'$ is coprime to $b_{12}$.
We know from the above that $b_{12}$ divides $y(b_{33}^{(1)}-b_{11}^{(1)}-yb_{31}^{(1)})$
for any $y\in R$. In particular, choosing $y=\lambda,\lambda',\lambda-\lambda'$,
respectively, we obtain $b_{31}^{(1)}(\lambda-\lambda')\in(b_{12})$,
hence $b_{31}^{(1)}\in(b_{12})$ and $b_{33}^{(1)}-b_{11}^{(1)}\in(b_{12})$.
This proves the claim.

By Claim~\ref{Claim I} there exist elements $\alpha,\beta\in R$
such that
\[
b_{33}^{(1)}-b_{11}^{(1)}=\alpha b_{12}\quad\text{and}\quad b_{31}^{(1)}=\beta b_{12}.
\]
Let
\[
B_{2}=(b_{ij}^{(2)})=(1+E_{21}(-\alpha+\beta))(1+E_{31})B_{1}(1+E_{31})^{-1}(1+E_{21}(-\alpha+\beta))^{-1}.
\]
Then $b_{12}^{(2)}=b_{32}^{(2)}=b_{12}$ and $b_{13}^{(2)}=b_{31}^{(2)}=0$
so that
\[
B_{2}=\begin{pmatrix}b_{11}^{(2)} & b_{12} & 0\\
b_{21}^{(2)} & b_{22}^{(2)} & b_{23}^{(2)}\\
0 & b_{12} & b_{33}^{(2)}
\end{pmatrix}.
\]
Moreover, let
\[
B'_{2}=(1-E_{31})B_{2}(1-E_{31})^{-1}=\begin{pmatrix}b_{11}^{(2)} & b_{12} & 0\\
b_{23}^{(2)}+b_{21}^{(2)} & b_{22}^{(2)} & b_{23}^{(2)}\\
b_{33}^{(2)}-b_{11}^{(2)} & 0 & b_{33}^{(2)}
\end{pmatrix}
\]
and
\[
B''_{2}=(1-E_{33})B_{2}(1-E_{33})^{-1}=\begin{pmatrix}b_{33}^{(2)} & b_{12} & 0\\
b_{23}^{(2)}+b_{21}^{(2)} & b_{22}^{(2)} & b_{21}^{(2)}\\
b_{11}^{(2)}-b_{33}^{(2)} & 0 & b_{11}^{(2)}
\end{pmatrix}.
\]
We will now show that $B_{2}$ has the property that $b_{12}\mid b_{ij}^{(2)}$
for all $i\neq j$ and $b_{12}\mid(b_{ii}^{(2)}-b_{jj}^{(2)})$ for
all $1\leq i,j\leq3$. This follows from the following fact applied
to the matrices $B'_{2}$ and $B''_{2}$.
\begin{claim}
\label{Claim II}Suppose that $C=(c_{ij})\in\M_{n}(R)$ satisfies
$(*)$ and that $c_{32}=0$. Then $c_{12}\mid c_{ij}$ for all $i,j$
such that $(i,j)\neq(2,1)$ and $i\neq j$, and $c_{12}\mid(c_{ii}-c_{jj})$
for all $1\leq i,j\leq3$.
\end{claim}
To prove the claim, let $x\in R$ and
\[
X=(x_{ij})=(1+E_{32}x)C(1+E_{32}x)^{-1}.
\]
Then
\begin{align*}
x_{12} & =c_{12},\\
x_{32} & =x(c_{22}-c_{33}-xc_{23}),
\end{align*}
and by Lemma~\ref{lem:row-column} applied to the second column in
$X$ we conclude that $c_{12}$ divides $x(c_{22}-c_{33}-xc_{23})$
for any $x\in R$. By $(*)$ and the same argument as in the proof
of Claim~\ref{Claim I} we obtain
\[
c_{12}\mid(c_{22}-c_{33})\quad\text{and}\quad c_{12}\mid c_{23}.
\]
Next, for $y\in R$ let
\[
Y=(y_{ij})=(1+E_{13}y)C(1+E_{13}y)^{-1}.
\]
Then
\begin{align*}
y_{12} & =c_{12},\\
y_{13} & =y(c_{33}-c_{11}-yc_{31}),
\end{align*}
and by Lemma~\ref{lem:row-column} applied to the first row in $Y$
and the same argument as for the matrix $X$ (that is, using $(*)$
and the same argument as in the proof of Claim~\ref{Claim I}) we
obtain
\[
c_{12}\mid(c_{33}-c_{11})\quad\text{and}\quad c_{12}\mid c_{31},
\]
whence also $c_{12}\mid(c{}_{22}-c_{11})$. This proves Claim~\ref{Claim II}
for $C$.

Applying Claim~\ref{Claim II} to the matrices $B'_{2}$ and $B''_{2}$,
respectively, we conclude that $B_{2}$ has the property that $b_{12}\mid b_{ij}^{(2)}$
for all $i\neq j$ and $b_{12}\mid(b_{ii}^{(2)}-b_{jj}^{(2)})$ for
all $1\leq i,j\leq3$. Since $B_{2}$ is similar to $B$ (and $B$
is similar to $A$), we have
\[
(b_{ii}-b_{jj},b_{ij}\mid i\neq j,1\leq i,j\leq3)=(1),
\]
so $b_{12}$ must be a unit. This proves the proposition.
\end{proof}
We now use Proposition~\ref{prop:3x3-normalform} to prove the corresponding
result for matrices in $\M_{n}(R)$ for all $n\geq3$. More precisely,
we have
\begin{thm}
\label{thm:LF-normalform}Let $A\in\M_{n}(R)$ with $n\geq3$, be
non-scalar. Then $A$ is similar to a matrix $B=(b_{ij})\in\M_{n}(R)$
such that $b_{12}\mid b_{ij}$ for all $i\neq j$ and $b_{12}\mid(b_{ii}-b_{jj})$
for all $1\leq i,j\leq n$. Moreover, $B$ may be chosen with $b_{ij}=0$
for all $i,j$ such that $j\geq i+2$ and $1\leq i\leq n-2$.\end{thm}
\begin{proof}
As in the proof of Proposition~\ref{prop:3x3-normalform}, we may
assume that
\[
(a_{ii}-a_{jj},a_{ij}\mid i\neq j,1\leq i,j\leq n)=(1),
\]
and choose a matrix $B$ satisfying the following condition

\MyQuote{The matrix $B=(b_{ij})$ is similar to $A$, $(b_{12},2)=(1)$, $b_{1j}=0$ for $j\geq 3$, the entry $b_{12}$ has the smallest number of distinct prime factors among all the matrices similar to $A$ and among all matrices with these properties $B$ is such that $b_{12}$ has the minimal number of not necessarily distinct prime factors.}If
for some $i,j$ the entry $b_{12}$ does not divide $b_{ii}-b_{jj}$,
then $b_{12}$ does not divide $b_{11}-b_{vv}$ for some $v$. If
$v\geq4$ let $W_{v}=(w_{ij}^{(v)})\in\GL_{n}(R)$ be any permutation
matrix such that $w_{11}^{(v)}=w_{22}^{(v)}=1$, $w_{v3}^{(v)}=1$
and $w_{3v}^{(v)}=1$. Then $W_{v}BW_{v}^{-1}$ has $(1,2)$ entry
equal to $b_{12}$ and $(3,3)$ entry equal to $b_{vv}$, so we may
assume that $b_{12}$ does not divide $b_{11}-b_{22}$ or $b_{11}-b_{33}$.
Consider the submatrix
\[
B_{0}=(b_{ij})_{1\leq i,j\leq3}
\]
of $B$ and note that any similarity $B_{0}\mapsto g^{-1}B_{0}g$
for $g\in\GL_{3}(R)$ may be achieved by $B\mapsto(g\oplus I_{n-3})B(g\oplus I_{n-3})^{-1}$.
By the minimality property of $b_{12}$ expressed in $(*)$ and the
argument in the proof of Proposition~\ref{prop:3x3-normalform} applied
to $B_{0}$ we conclude that $b_{12}$ divides both $b_{11}-b_{22}$
and $b_{11}-b_{33}$, which is a contradiction. Thus
\[
b_{12}\mid(b_{ii}-b_{jj})\text{ for all }1\leq i,j\leq n\quad\text{and}\quad b_{12}\mid b_{ij}\text{ for all }i\neq j,\,1\leq i,j\leq3.
\]
Similarly, for any $4\leq v\leq n$ the matrix $W_{v}BW_{v}^{-1}$
has $(3,1)$ entry equal to $b_{v1}$, so by $(*)$ and the argument
in the proof of Proposition~\ref{prop:3x3-normalform} applied to
$B_{0}$ we conclude that $b_{12}\mid b_{v1}$. Hence
\[
b_{12}\mid b_{v1}\text{ for all }4\leq v\leq n.
\]
Furthermore, by $(*)$ and Lemma~\ref{lem:row-column} applied to
the second column in $B$, we see that
\[
b_{12}\mid b_{i2}\text{ for all }i\neq2.
\]
Let $1\leq u,v\leq n$ be such that $u\geq3$ and $v\neq u$. For
$x\in R$ let
\[
X_{u}=(x_{ij}^{(u)})=(1+E_{u2})B(1+E_{u2})^{-1},
\]
so that $x_{v2}^{(u)}=b_{v2}-b_{vu}$ and in particular $x_{12}^{(u)}=b_{12}$.
By $(*)$ and Lemma~\ref{lem:row-column} applied to the second column
in $X_{u}$ we see that $b_{12}\mid x_{v2}^{(u)}$ and since $b_{12}\mid b_{v2}$
we conclude that $b_{12}\mid b_{vu}$. Hence
\[
b_{12}\mid b_{vu}\text{ for all }u\geq3,\, v\neq u.
\]
 We have thus shown that $B$ has the property that $b_{12}\mid b_{ij}$
for all $i\neq j$ and $b_{12}\mid(b_{ii}-b_{jj})$ for all $1\leq i,j\leq n$.

For the second statement we follow \cite[III,~2]{Newman}. Conjugating
$B$ by $1_{2}\oplus M_{3j}\in\GL_{n}(R)$ for a suitable $M_{3j}\in\GL_{n-2}(R)$
(cf.~the proof of Lemma~\ref{lem:row-column}), we can replace $B$
by a matrix $B_{1}$ in which the first row equals that of $B$ and
whose $(2,j)$ entries are zero whenever $j\geq4$. Conjugating $B_{1}$
by $1_{3}\oplus M_{4j}\in\GL_{n}(R)$ for a suitable $M_{4j}\in\GL_{n-3}(R)$,
we can replace $B_{1}$ by a matrix $B_{2}$ in which the first two
rows equal those of $B_{1}$ and whose $(3,j)$ entries are zero whenever
$j\geq5$. Proceeding inductively in this way, we obtain a matrix
$C=(c_{ij})$ similar to $B$ such that $c_{12}=b_{12}$ and $c_{ij}=0$
for $i,j$ such that $j\geq i+2$ and $1\leq i\leq n-2$. But since
$B\equiv b_{11}1_{n}\bmod{(b_{12})}$ we also have $C\equiv b_{11}1_{n}\bmod{(b_{12})}$,
so $C$ has the desired form.
\end{proof}
Using Theorem~\ref{thm:LF-normalform} it is now easy to prove the
following result. The following proof is entirely analogous to that
of Laffey and Reams for $R=\Z$.
\begin{prop}
Let $A\in\M_{n}(R)$, $n\geq3$ have trace zero, and suppose that
for every $\mfp\in\Specm R$ and every $a\in R/\mfp$, $a\neq0$ we
have $A_{\mfp}\neq a1_{n}$. Then $A$ is similar to a matrix $B=(b_{ij})\in\M_{n}(R)$
where $b_{ii}=0$ for all $1\leq i\leq n$.\end{prop}
\begin{proof}
If $A_{\mfp}=0$ for some $\mfp$, we can write $A=mA'$, where $m\in R$
and $A'$ is such that for every $\mfp\in\Specm R$ and every $a\in R/\mfp$
we have $A_{\mfp}'\neq a1_{n}$. Since $A'$ must be non-scalar Theorem~\ref{thm:LF-normalform}
implies that $A'$ is similar to a matrix $A''=(a_{ij}'')$ such that
$a_{12}''\mid a_{ij}''$ for all $i\neq j$ and $a_{12}''\mid(a_{ii}''-a_{jj}'')$
for all $1\leq i,j\leq n$. Since $A''$ satisfies $A_{\mfp}''\neq a1_{n}$
for any $\mfp\in\Specm R$ and $a\in R/\mfp$, the entry $a_{12}''$
must be a unit. We may therefore assume without loss of generality
that $A=A''$, so that in particular $a_{12}$ is a unit.

We now prove that $A$ is similar to a matrix with zero diagonal by
induction on $n$. If $n=2$, the matrix
\[
(1+E_{21}a_{11}a_{12}^{-1})A(1+E_{21}a_{11}a_{12}^{-1})^{-1}
\]
has zero diagonal. If $n>2$, conjugating $A$ by a matrix of the
form $1+\alpha E_{n1}$, $\alpha\in R$, we may assume that $a_{n2}=1$,
and then conjugating $A$ by a matrix of the form $1+\beta E_{21}$,
$\beta\in R$, we may further assume that $a_{11}=0$. Thus we may
assume that $A$ is of the form
\[
\begin{pmatrix}0 & x\\
y^{T} & A_{1}
\end{pmatrix},
\]
where $x,y\in R^{n-1}$, $A_{1}=(a_{ij}^{1})\in\M_{n-1}(R)$ with
$a_{n-1,1}^{1}=1$ and $\Tr(A_{1})=0$. By Theorem~\ref{thm:LF-normalform}
$A_{1}$ is similar to a matrix $A_{2}=(a_{ij}^{2})$ such that $a_{12}^{2}\mid a_{ij}^{2}$
for all $i\neq j$ and $a_{12}^{2}\mid(a_{ii}^{2}-a_{jj}^{2})$ for
all $1\leq i,j\leq n$. Since $(A_{2})_{\mfp}\neq a1_{n}$ for all
$\mfp\in\Specm R$ and $a\in R/\mfp$, the entry $a_{12}^{2}$ must
be a unit. So by induction there exists a $Q\in\GL_{n-1}(R)$ such
that $QA_{1}Q^{-1}=B_{1}$ is a matrix with zeros on the diagonal.
But then
\[
B=(1_{1}\oplus Q)A(1_{1}\oplus Q)^{-1}
\]
 has the desired form.
\end{proof}
A matrix in $\M_{n}(R)$ satisfying the conditions on the matrix $B$
in Theorem~\ref{thm:LF-normalform} will be said to be in \emph{Laffey-Reams
form}.

\section{\label{sec:Proof-Main}Proof of the main result}

In this section we give a proof of our main theorem on commutators,
Theorem~\ref{thm:Main}. We first prove a couple of lemmas used in
the proof.
\begin{lem}
\label{lem:GCD}Let $R$ be a PID. Then the following holds:
\begin{enumerate}
\item \label{enu:GCD-lemma abx}Let $a,b\in R$ be such that $(a,b)=(1)$,
and let $S$ be a finite set of maximal ideals of $R$. Then there
exists an $x\in R$ such that for all $\mfp\in S$ we have $a+bx\notin\mfp$.
\item \label{enu:GCD-lemma pqt}Let $\alpha,\beta\in R$ be such that $(\alpha,\beta)=(1)$.
Suppose that $\mfp$ is a maximal ideal of $R$ such that $|R/\mfp|\geq3$.
Then for every finite set $S$ of maximal ideals of $R$ such that
$\mfp\notin S$ there exists a $t\in R$ such that $t\notin\mfp$,
$t\in\mfq$ for all $\mfq\in S\setminus\{\mfp\}$ and $\alpha t+\beta\notin\mfp$.
\item \label{enu:GCD-lemma abc}Let $a,b,c\in R$ be such that $(a,b,c)=(1)$,
$(a,b)\neq(1)$ and $(a,c)\neq(1)$. Then there exists an $x\in R$
such that $(a+cx,b-ax)=(1)$.
\end{enumerate}
\end{lem}
\begin{proof}
To prove \ref{enu:GCD-lemma abx}, take $x$ to be a generator of
the product
\[
\prod_{\substack{\mfp\in S\\
a\notin\mfp
}
}\mfp
\]
and let $x=1$ if there is no $\mfp\in S$ such that $a\notin\mfp$.
Let $\mfp\in S$ be such that $a\in\mfp$. If $a+bx\in\mfp$, then
$ $$bx\in\mfp$ and since $(a,b)=(1)$ we have $x\in\mfp$, which
contradicts the definition of $x$. On the other hand, let $\mfp\in S$
be such that $a\notin\mfp$. If $a+bx\in\mfp$, then by the definition
of $x$ we have $bx\in\mfp$, so $a\in\mfp$, which is a contradiction.
Thus in either case, $a+bx\notin\mfp$.

Next, we prove \ref{enu:GCD-lemma pqt}. Since $|R/\mfp|\geq3$ there
exist two elements $r_{1},r_{2}\in R\setminus\mfp$ such that $r_{1}-r_{2}\notin\mfp$.
Let $s\in R$ be such that
\[
(s)=\prod_{\mfq\in S\setminus\{\mfp\}}\mfq.
\]
Then for $i=1,2$ we have $r_{i}s\notin\mfp$ and $r_{i}s\in\mfq$
for all $\mfq\in S\setminus\{\mfp\}$. Furthermore, if $\alpha r_{i}s+\beta\in\mfp$
for $i=1,2$, then $\alpha\in\mfp$ and $\beta\in\mfp$, contradicting
the hypothesis $(\alpha,\beta)=(1)$. Thus we may assume that $\alpha r_{1}s+\beta\notin\mfp$,
and $t=r_{1}s$ yields the desired element.

We now prove \ref{enu:GCD-lemma abc}. We first show that $a+cx$
and $b-ax$ are relatively prime as elements of $R[x]$, that is,
that none of them is a multiple of the other. Indeed, if $a+cx=m(b-ax)$
for some $m\in R$, then $a=mb$ and $c=-ma$ so $(1)=(a,b,c)=(mb,b,-m^{2}b)=(b)$,
which is impossible since $(a,b)\neq(1)$. Similarly, if $n(a+cx)=b-ax$
for some $n\in R$, then $b=na$ and $a=-nc$ so $(1)=(a,b,c)=(-nc,-n^{2}c,c)=(c)$,
which is impossible. Let $K$ be the field of fractions of $R$. Since
$a+cx$ and $b-ax$ are relatively prime as elements of $R[x]$, they
are relatively prime as element of $K[x]$. Thus there exists $f_{0},g_{0}\in K[x]$
such that $(a+cx)f_{0}+(b-ax)g_{0}=1$, and so there exists some $f,g\in R[x]$
such that
\begin{equation}
(a+cx)f+(b-ax)g=D\in R\setminus\{0\}.\label{eq:fgD}
\end{equation}
Let $S$ be the set of maximal ideals dividing $(D)$. By \ref{enu:GCD-lemma abx}
we can choose $x\in R$ such that for all $\mfp\in S$ we have $a+cx\notin\mfp$.
Now if $\mfp$ is a maximal ideal of $R$ such that $a+cx\in\mfp$
and $b-ax\in\mfp$, then $D\in\mfp$, by (\ref{eq:fgD}), and so $a+cx\notin\mfp$;
a contradiction. Thus there is no $\mfp\in\Specm R$ such that $a+cx\in\mfp$
and $b-ax\in\mfp$, that is, $a+cx$ and $b-ax$ are relatively prime.
\end{proof}
The following result is the Chinese remainder theorem for centralisers
of matrices over quotients of $R$. It will be used at a crucial step
in our proof of Theorem~\ref{thm:Main}.
\begin{lem}
\label{lem:Centr-product}Let $X\in\M_{n}(R)$ and let $\mfp_{1},\dots,\mfp_{\nu}$,
$\nu\in\N$ be maximal ideals in $R$. Then the map
\begin{align*}
C_{\M_{n}(R/(\mfp_{1}\cdots\mfp_{\nu}))}(X_{(\mfp_{1}\cdots\mfp_{\nu})}) & \longrightarrow\prod_{i=1}^{\nu}C_{\M_{n}(R/\mfp_{i})}(X_{\mfp_{i}})\\
g & \longmapsto(g_{\mfp_{1}},\dots,g_{\mfp_{\nu}}),
\end{align*}
is an isomorphism. \end{lem}
\begin{proof}
Let $\mathcal{C}=C_{\M_{n}(R/(\mfp_{1}\cdots\mfp_{\nu}))}(X_{(\mfp_{1}\cdots\mfp_{\nu})})$.
Then $\mathcal{C}$ is a module over $R$. By the Chinese remainder
theorem we have an isomorphism $R/(\mfp_{1}\cdots\mfp_{\nu})\rightarrow\prod_{i=1}^{\nu}R/\mfp_{i}$
given by $a\mapsto(a_{\mfp_{1}},\dots,a_{\mfp_{\nu}})$, and tensoring
this by $\mathcal{C}$ yields
\begin{align*}
\mathcal{C} & \cong R/(\mfp_{1}\cdots\mfp_{\nu})\otimes_{R}\mathcal{C}\cong\big(\prod_{i=1}^{\nu}R/\mfp_{i}\big)\otimes_{R}\mathcal{C}\cong\prod_{i=1}^{\nu}(R/\mfp_{i}\otimes_{R}\mathcal{C})\\
 & \cong\prod_{i=1}^{\nu}C_{\M_{n}(R/\mfp_{i})}(X_{\mfp_{i}}).
\end{align*}
Tracking the maps shows that the effect of the above isomorphisms
on elements is given by
\[
g\longmapsto1\otimes g\longmapsto(1_{\mfp_{1}},\dots,1_{\mfp_{\nu}})\otimes g\longmapsto(1_{\mfp_{1}}\otimes g,\dots,1_{\mfp_{\nu}}\otimes g)\longmapsto(g_{\mfp_{1}},\dots,g_{\mfp_{\nu}}).
\]

\end{proof}
We now give the proof of our main theorem. Note that our proof in
the case $n=2$ is different from the case $n\geq3$, and that for
$n=2$, while our argument is not the shortest possible, yields the
stronger result that any $A\in\M_{2}(R)$ with trace zero can be written
as $A=[X,Y]$ for some $X,Y\in\M_{2}(R)$ and $X$ regular.
\begin{thm}
\label{thm:Main}Let $R$ be a PID and let $A\in\M_{n}(R)$ be a matrix
with trace zero. Then $A=[X,Y]$ for some $X,Y\in\M_{n}(R)$. \end{thm}
\begin{proof}
For $n=1$ the result is trivial. First assume that $n=2$. By taking
out a suitable factor we may assume that the matrix
\[
A=\begin{pmatrix}a & b\\
c & -a
\end{pmatrix}
\]
satisfies $(a,b,c)=(1)$. Let $X=\begin{pmatrix}0 & 1\\
x_{1} & x_{2}
\end{pmatrix}\in\M_{2}(R)$. By Lemma~\ref{lem:reg-triang} the matrix $X$ is regular so it
is regular mod $\mfp$ for every maximal ideal $\mfp$ of $R$. Furthermore,
\[
\Tr(XA)=bx_{1}-ax_{2}+c,
\]
so if $(a,b)=(1)$ we can find $x_{1}$ and $x_{2}$ such that $\Tr(XA)=0$,
and Proposition~\ref{prop:Criterion} implies that $A=[X,Y]$, for
some $Y\in\M_{2}(R)$. Similarly, the transpose $X^{T}$ of $X$ is
also regular, and
\[
\Tr(X^{T}A)=cx_{1}-ax_{2}+b,
\]
so if $(a,c)=(1)$ we can find $x_{1}$ and $x_{2}$ such that $\Tr(X^{T}A)=0$,
and so $A=[X^{T},Y]$, for some $ $$Y\in\M_{2}(R)$. Hence, in case
$(a,b)=(1)$ or $(a,c)=(1)$ we are done. Assume therefore that $(a,b)\neq(1)$
and $(a,c)\neq(1)$. If we let $T=1+xE_{12}\in\M_{2}(R)$ for some
$x\in R$, we have
\[
TAT^{-1}=\begin{pmatrix}a+cx & b-ax-x(a+cx)\\
c & -a-cx
\end{pmatrix}.
\]
Now $a+cx$ and $b-ax-x(a+cx)$ are relatively prime if and only if
$a+cx$ and $b-ax$ are relatively prime. By Lemma~\ref{lem:GCD}\,\ref{enu:GCD-lemma abc}
we can choose $x\in R$ such that $(a'+c'x,b'+a''x)=(1)$, and hence
such that the $(1,1)$ and $(1,2)$ entries in $TAT^{-1}$ are relatively
prime. As we have already seen, this means that we can find $x_{1}$
and $x_{2}$ such that $\Tr(XTAT^{-1})=0$, so Proposition~\ref{prop:Criterion}
yields $A=[T^{-1}XT,Y]$ for some $Y\in\M_{2}(R)$.

Assume now that $n\geq3$. If $A$ is a scalar matrix we obviously
have $\Tr(J_{n}(0)^{r}A)=0$ for all $r\geq0$, so Proposition~\ref{prop:Criterion}
yields the desired conclusion. We may therefore henceforth assume
that $A$ is non-scalar. Write $A=(a_{ij})$ for $1\leq i,j,\leq n$.
By Theorem~\ref{thm:LF-normalform} we may assume that $A$ is in
Laffey-Reams form. If $d\in R$ is such that $(a_{ij},a_{ii}-a_{jj}\mid i\neq j,1\leq i,j\leq n)=(d)$,
we can write $A=dA'$ where $A'=(a'_{ij})\in\M_{n}(R)$ is in Laffey-Reams
form and $(a_{11}',a_{12}')=(1)$. It thus suffices to assume that
$A=A'$ so that $(a_{11},a_{12})=(1)$, $a_{12}\mid a_{ij}$ for $i\neq j$,
$a_{12}\mid(a_{ii}-a_{jj})$ for $1\leq i,j\leq n$, and $a_{ij}=0$
for $j\geq i+2$. Let $k=\lfloor n/2\rfloor$. For $x,y,q\in R$ define
the matrix $X=(x_{ij})\in\M_{n}(R)$ by
\[
(x_{ij})=\begin{cases}
x_{ii}=-y & \text{for }i=2,4,\dots,2k,\\
x_{21}=x,\\
x_{31}=q,\\
x_{j,j-1}=1 & \text{for }j=3,4,\dots,n,\\
x_{ij}=0 & \text{otherwise}.
\end{cases}
\]
Recall that for any $B=(b_{ij})\in\M_{n}(R)$ we write $c(B)=\sum_{i=1}^{k}b_{2i,2i}$.We
have
\[
\Tr(XA)=xa_{12}+a_{23}+\dots+a_{n-1,n}-yc(A).
\]
We claim that $\Tr(XA)=0$ implies that $\Tr(X^{r}A)=0$ for all $r\geq0$.
To see this, observe that the matrix $X^{2}+yX$ is lower triangular
and its $(i,j)$ entry is $0$ if $j\geq i-1$. Since $\Tr(E_{ij}A)=0$
if $j<i-1$ (since $a_{ij}=0$ for $j\geq i+2$), it follows that
$\Tr((X^{2}+yX)A)=0$, so if $\Tr(XA)=0$ we get $\Tr(X^{2}A)=0$.
More generally, using the fact that $X$ is lower triangular, we have
$\Tr((X^{r}+yX^{r-1})A)=0$, and working inductively we get $\Tr(X^{r}A)=0$
for all $r\geq0$.

Assume for the moment that $a_{12}\mid c(A)$ and let $M=1-c(A)a_{12}^{-1}E_{21}\in\M_{n}(R)$.
Then
\[
c(MAM^{-1})=0,
\]
so Proposition~\ref{prop:Criterion} together with (\ref{eq:Tr_c(A)})
and the fact that $P_{n}$ is regular imply $MAM^{-1}=[P_{n},Y]$,
for some $Y\in\M_{n}(R)$. Thus in this case $A=[M^{-1}P_{n}M,M^{-1}YM]$,
so we may henceforth assume that
\begin{equation}
a_{12}\nmid c(A).\label{eq:a12-notdiv-c(A)}
\end{equation}
We now show that there exist elements $x,y\in R$ with $(x,y)=(1)$
and such that $\Tr(XA)=0$. To this end, consider the equation
\[
xa_{12}+a_{23}+\dots+a_{n-1,n}=yc(A),\qquad x,y\in R.
\]
Since $a_{12}$ divides $a_{23},\dots,a_{n-1,n}$, this may be written
\begin{equation}
a_{12}(x+l)=yc(A),\label{eq:Diophant-xy}
\end{equation}
for some $l\in R$. Let $d\in R$ be a generator of $(a_{12},c(A))$.
Then (\ref{eq:Diophant-xy}) is equivalent to
\begin{align*}
x & =hc(A)d^{-1}-l\\
y & =ha_{12}d^{-1},
\end{align*}
for any $h\in R$. Choose $h$ to be a generator of the product of
all maximal ideals $\mfp$ of $R$ such that $a_{12}d^{-1}\in\mfp$
and $l\notin\mfp$ (and let $h=1$ if no such $\mfp$ exist). Suppose
that $(x,y)\in(p)$ for some prime element $p\in R$. Then $y\in(p)$
and so $a_{12}d^{-1}\in(p)$ or $h\in(p)$. If $a_{12}d^{-1}\in(p)$
and $l\not\in(p)$, then $h\in(p)$, so $x\notin(p)$. If $a_{12}d^{-1}\in(p)$
and $l\in(p)$, then $h\notin(p)$ and since $(a_{12}d^{-1},c(A)d^{-1})=(1)$
we have $x\notin(p)$. Furthermore, if $h\in(p)$ then $l\notin(p)$
so $x\notin(p)$. Thus $(x,y)=(1)$. If $y$ is a unit then $a_{12}d^{-1}$
must be a unit, and so $a_{12}\mid c(A)$, contradicting (\ref{eq:a12-notdiv-c(A)}).
Thus $y$ is not a unit, and so $x^{2}a_{12}\notin(ya_{12})$. Since
$a_{12}$ divides each of $a_{11}-a_{22}$, $a_{21}$, $a_{31}$ and
$a_{32}$, we have $xy(a_{11}-a_{22})-y^{2}(a_{21}+ya_{31}+xa_{32})\in(ya_{12})$.
Thus, we must have
\begin{equation}
x^{2}a_{12}+xy(a_{11}-a_{22})-y^{2}(a_{21}+ya_{31}+xa_{32})\neq0.\label{eq:not-zero}
\end{equation}
From now on let $x$ and $y$ be as above, so that $(x,y)=(1)$ and
$\Tr(XA)=0$. Next, we specify the entry $q$ in $X$.

Let $S_{0}$ be the set of maximal ideals $\mfp$ of $R$ such that
$x^{2}a_{12}+xy(a_{11}-a_{22})-y^{2}(a_{21}+ya_{31}+xa_{32})\in\mfp$,
and let
\[
S=S_{0}\cup\{\mfp\in\Specm R\mid|R/\mfp|=2\}.
\]
Note that $S$ is a finite set because of (\ref{eq:not-zero}) together
with the fact that for any PID $R'$ (or any Dedekind domain), there
are only finitely many $\mfp\in\Specm R'$ such that $|R'/\mfp|=2$.
By Lemma~\ref{lem:GCD}\,\ref{enu:GCD-lemma abx} (with $r=1$)
we can thus choose $q\in R$ such that
\[
x+qy\notin\mfp,\quad\text{for all }\mfp\in S.
\]
Assume from now on that $q$ has been chosen in this way. Let $V$
be the set of maximal ideals of $R$ such that $x+qy\in\mfp$, that
is,
\[
V=\{\mfp\in\Specm R\mid x+qy\in\mfp\}.
\]
By the choice of $q$ we thus have in particular that
\begin{equation}
\mfp\in V\Longrightarrow x^{2}a_{12}+xy(a_{11}-a_{22})-y^{2}(a_{21}+ya_{31}+xa_{32})\notin\mfp.\label{eq:pinV-polynotinp}
\end{equation}
Note that for every $\mfp\in V$ we have $y\notin\mfp$ since $(x,y)=(1)$.
Note also that $S\cap V=\varnothing$.

We claim that $X_{\mfp}\in\M_{n}(R/\mfp)$ is regular for every maximal
ideal $\mfp$ not in $V$. To show this, let $\mfp\in(\Specm R)\setminus V$
and let
\[
M=\begin{pmatrix}x+qy & 0\\
q & 1
\end{pmatrix}\oplus1_{n-2}\in\M_{n}(R).
\]
Since $x+qy\not\in\mfp$ the image $M_{\mfp}\in\M_{n}(R/\mfp)$ of
$M$ is invertible and, letting $y_{\mfp}$ denote the image of $y$
in $R/\mfp$, we have
\[
M_{\mfp}X_{\mfp}M_{\mfp}^{-1}=(m_{ij})=\begin{cases}
m_{ii}=-y_{\mfp} & \text{for }i=2,4,\dots,2k,\\
m_{j,j-1}=1 & \text{for }j=2,3,\dots,n,\\
m_{ij}=0 & \text{otherwise}.
\end{cases}
\]
It follows from Lemma~\ref{lem:reg-triang} that $M_{\mfp}X_{\mfp}M_{\mfp}^{-1}$
is regular, and thus $X_{\mfp}$ is regular.

By our choice of $q$ we have $\mfp\notin V$ if $\mfp\in S$, so
$X_{\mfp}$ is regular for any $\mfp\in S$, and $S$ is non-empty.
By Proposition~\ref{prop:Reg-mod-m} we have that $X$ is regular
as an element in $\M_{n}(F)$, where $F$ is the field of fractions
of $R$. By our choice of $x$ and $y$ we have $\Tr(X^{r}A)=0$ for
$r=0,1,\dots,n-1$, so Proposition~\ref{prop:LF-criterion-fields}
implies that we can write $A=[X,Q]$, for some $Q\in\M_{n}(F)$. Clearing
denominators in $Q$ we find that there exists a non-zero element
$m_{0}\in R$ such that $m_{0}A\in[X,\M_{n}(R)]$. We now highlight
a step which we will refer to in the following:

\MyQuote{Let $m\in R$ be such that it has the minimal number of (not necessarily distinct) prime factors among all $m'\in R$ such that $m'A\in[X,\M_n(R)]$, and let $Q\in\M_n(R)$ be such that $mA=[X,Q]$.} We
show that the only maximal ideals containing $m$ are those in $V$.
Suppose that $\mfp=(p)\in(\Specm R)\setminus V$ and that $m\in\mfp$.
Then $0=[X_{\mfp},Q{}_{\mfp}]$, and since $X_{\mfp}$ is regular
there exists a polynomial $f\in R[T]$ such that $Q=f(X)+pQ'$ for
some $Q'\in\M_{n}(R)$, so $mA=[X,f(X)+pQ']=[X,pQ']$ and thus $mp^{-1}A=[X,Q']$,
which contradicts $(*)$. Thus, if $m\in\mfp$ for some $\mfp\in\Specm R$,
then we must have $\mfp\in V$. Let $\mfp_{1},\mfp_{2},\dots,\mfp_{\nu}$,
$\nu\in\N$ be the elements of $V$ such that $m\in\mfp_{i}$. For
each $\mfp_{i}$, choose a generator $p_{i}\in R$, so that $\mfp_{i}=(p_{i})$,
for $i=1,\dots,\nu$. We then have
\[
(m)=(p_{1}^{e_{1}}p_{2}^{e_{2}}\cdots p_{\nu}^{e_{\nu}}),
\]
for some $e_{i}\in\N$, $1\leq i\leq\nu$.

The strategy is now to show that $X$ can be replaced by a matrix
$X_{1}$ which is regular mod $\mfp$ for every $\mfp\in V$. Let
\[
N=1+qE_{21}\in\M_{n}(R).
\]
For ease of calculation we will consider the matrices
\[
A_{0}=NAN^{-1},\quad X_{0}=NXN^{-1},\quad Q_{0}=NQN^{-1}.
\]
Let $\mfp\in V$ be any of the ideals $\mfp_{1},\mfp_{2},\dots,\mfp_{\nu}$.
We have
\begin{equation}
(X_{0})_{\mfp}=\begin{pmatrix}0 & 0\\
0 & W_{\mfp}
\end{pmatrix}=(0)\oplus W_{\mfp},\label{eq:Wp}
\end{equation}
where $W_{\mfp}\in\M_{n-1}(R/\mfp)$ is regular. We wish to determine
the dimension of the centraliser
\[
C(\mfp):=C_{\M_{n}(R/\mfp)}((X_{0})_{\mfp}).
\]
Since $(x,y)=(1)$, we have $y_{\mfp}\neq0$, so the Jordan form of
$(X_{0})_{\mfp}$ is
\[
J_{k}(-y_{\mfp})\oplus J_{n-k-1}(0)\oplus J_{1}(0),
\]
where $k=\lfloor n/2\rfloor$, as before. We have an isomorphism of
$R/\mfp$-vector spaces
\[
C(\mfp)\cong C_{\M_{k}(R/\mfp)}(J_{k}(-y_{\mfp}))\oplus C_{\M_{n-k}(R/\mfp)}(J_{k}(0)\oplus J_{1}(0)).
\]
Since $\dim C_{\M_{k}(R/\mfp)}(J_{k}(-y_{\mfp}))=k$ it remains to
determine the dimension of $C_{\M_{n-k}(R/\mfp)}(J_{n-k-1}(0)\oplus J_{1}(0))$.
A matrix
\[
H=\begin{pmatrix}H_{11} & H_{12}\\
H_{21} & H_{22}
\end{pmatrix}\in\M_{n-k}(R/\mfp),
\]
where $H_{11}$ is a $(n-k-1)\times(n-k-1)$ block, $H_{22}$ is a
$1\times1$ block, and the other blocks are of compatible sizes, commutes
with $J_{n-k-1}(0)\oplus J_{1}(0)$ if and only if
\[
H_{11}J_{n-k-1}(0)=J_{n-k-1}(0)H_{11},\quad H_{12}\in\begin{pmatrix}R/\mfp\\
0
\end{pmatrix},\quad H_{21}\in(0,R/\mfp).
\]
Hence $\dim C_{\M_{n-k}(R/\mfp)}(J_{n-k-1}(0)\oplus J_{1}(0))=n-k-1+1+1+1$,
and so
\[
\dim C(\mfp)=n+2,
\]
that is, $(X_{0})_{\mfp}$ is subregular (cf.~\cite{Springer-Steinberg}).
Next, we need the dimension of $(R/\mfp)[(X_{0})_{\mfp}]$ (the algebra
of polynomials in $(X_{0})_{\mfp}$ over the field $R/\mfp$). Since
$(R/\mfp)[(X_{0})_{\mfp}]\cong(0)\oplus(R/\mfp)[W_{\mfp}]$ and $W_{\mfp}$
is regular, we have $\dim(R/\mfp)[(X_{0})_{\mfp}]=n-1$.

We now find a basis for $C(\mfp)$. We know that $(R/\mfp)[(X_{0})_{\mfp}]$
is an $(n-1)$-dimensional subspace of $C(\mfp)$. Moreover, direct
verification shows that $E_{11}$ and $E_{12}+y_{\mfp}E_{13}$ are
in $C(\mfp)$. Let $\kappa=n+1-2\lfloor(n+1)/2\rfloor$, that is,
$\kappa$ is $0$ if $n$ is odd and $1$ if $n$ is even. Then we
also have
\[
E_{n1}+\kappa y_{\mfp}E_{n-1,1}\in C(\mfp).
\]
Since $(X_{0})_{\mfp}$ is lower triangular and the first column of
$(X_{0})_{\mfp}^{i}$ is $0$ for all $i\in\N$, the intersection
of $(R/\mfp)[(X_{0})_{\mfp}]$ with the $R/\mfp$-span $\langle E_{11},E_{12}+y_{\mfp}E_{13},E_{n1}+\kappa y_{\mfp}E_{n-1,1}\rangle$
is $0$. Since $\{E_{11},E_{12}+y_{\mfp}E_{13},E_{n1}+\kappa y_{\mfp}E_{n-1,1}\}$
is linearly independent, $\dim(R/\mfp)[(X_{0})_{\mfp}]=n-1$ and $\dim C(\mfp)=n+2$,
we must have
\begin{equation}
C(\mfp)=\langle(R/\mfp)[(X_{0})_{\mfp}],E_{11},E_{12}+y_{\mfp}E_{13},E_{n1}+\kappa y_{\mfp}E_{n-1,1}\rangle.\label{eq:Centr-span}
\end{equation}
We observe that the matrix $E_{n1}+\kappa yE_{n-1,1}\in\M_{n}(R)$,
whose image in $\M_{n}(R/\mfp)$ is $E_{n1}+\kappa y_{\mfp}E_{n-1,1}$,
satisfies
\begin{equation}
E_{n1}+\kappa yE_{n-1,1}\in C_{\M_{n}(R)}(X_{0}).\label{eq:E+kyE-inC(X)}
\end{equation}
Let $\mfa=\prod_{i=1}^{\nu}\mfp_{i}$, so that $\mfa=(p_{1}\cdots p_{\nu})$.
By (\ref{eq:Centr-span}) and Lemma~\ref{lem:Centr-product} we have
\begin{equation}
C_{\M_{n}(R/\mfa)}(X_{\mfa})=\langle(R/\mfa)[(X_{0})_{\mfa}],E_{11},E_{12}+y_{\mfa}E_{13},E_{n1}+\kappa y_{\mfa}E_{n-1,1}\rangle.\label{eq:Centr-span-severalprimes}
\end{equation}
Since $[X_{0},Q_{0}]=mA_{0}$ we have $([X_{0},Q_{0}])_{\mfa}=0$,
that is, $(Q_{0})_{\mfa}\in C_{\M_{n}(R/\mfa)}((X_{0})_{\mfa})$.
Hence, by (\ref{eq:Centr-span-severalprimes})
\[
Q_{0}=f(X_{0})+\alpha E_{11}+\beta(E_{12}+yE_{13})+\gamma(E_{n1}+\kappa yE_{n-1,1})+p_{1}\cdots p_{\nu}D,
\]
for some $\alpha,\beta,\gamma\in R$, $f(T)\in R[T]$ and $D\in\M_{n}(R)$.
Using (\ref{eq:E+kyE-inC(X)}) we get
\begin{align}
[X_{0},Q_{0}] & =[X_{0},f(X_{0})+\alpha E_{11}+\beta(E_{12}+yE_{13})\label{eq:X0-Q0}\\
 & \quad+\gamma(E_{n1}+\kappa yE_{n-1,1})+p_{1}\cdots p_{\nu}D]\nonumber \\
 & =[X_{0},\alpha E_{11}+\beta(E_{12}+yE_{13})+p_{1}\cdots p_{\nu}D]\nonumber \\
 & =[X_{0},Q_{1}],\nonumber
\end{align}
where
\[
Q_{1}:=\alpha E_{11}+\beta(E_{12}+yE_{13})+p_{1}\cdots p_{\nu}D.
\]
Let $i\in\N$ be such that $1\leq i\leq\nu$. If $(\alpha,\beta)\subseteq\mfp_{i}$
then $[X_{0},Q_{1}]\in p_{i}[X_{0},\M_{n}(R)]$ and so $mp_{i}^{-1}A\in[X,\M_{n}(R)]$,
contradicting $(*)$. Thus either $\alpha\notin\mfp_{i}$ or $\beta\notin\mfp_{i}$.
We show that the case where $\alpha\in\mfp_{i}$ and $\beta\notin\mfp_{i}$
cannot arise. Since $mA_{0}=[X_{0},Q_{0}]=[X_{0},Q_{1}]$, we have
$m\cdot\Tr(Q_{1}A_{0})=0$, whence $\Tr(Q_{1}A_{0})=0$. Together
with $\alpha\in\mfp_{i}$ and $\beta\notin\mfp_{i}$ this implies
that
\[
\Tr((E_{12}+yE_{13})A_{0})\in\mfp_{i}.
\]
Recalling that $A_{0}=NAN^{-1}$ we thus get
\[
-q^{2}a_{12}+q(a_{11}-a_{22})+a_{21}+ya_{31}-qya_{32}\in\mfp_{i}
\]
and, after multiplying by $y^{2}$,
\[
-q^{2}y^{2}a_{12}+qy^{2}(a_{11}-a_{22})+y^{2}(a_{21}+ya_{31}-qya_{32})\in\mfp_{i}.
\]
Since $\mfp_{i}\in V$ we have $qy\in-x+\mfp_{i}$ and so
\[
x^{2}a_{12}+xy(a_{11}-a_{22})-y^{2}(a_{21}+ya_{31}+xa_{32})\in\mfp_{i}.
\]
But by our choice of $q$ we have
\[
x^{2}a_{12}+xy(a_{11}-a_{22})-y^{2}(a_{21}+ya_{31}+xa_{32})\notin\mfp,
\]
for all $\mfp\in V$, which together with (\ref{eq:pinV-polynotinp})
yields a contradiction. Therefore we cannot have $\alpha\in\mfp_{i}$
and $\beta\notin\mfp_{i}$, so we must have $\alpha\notin\mfp_{i}$.
We have thus shown that
\[
\alpha\notin\mfp_{i},\quad\text{for all }i=1,\dots,\nu.
\]
By Lemma~\ref{lem:GCD}\,\ref{enu:GCD-lemma pqt} and our choice
of $S$ there exists a $t\in R$ such that
\begin{equation}
t\notin\mfp_{i}\quad\text{and}\quad\alpha t+y\notin\mfp_{i},\quad\text{for all }i=1,\dots,\nu.\label{eq:at+y}
\end{equation}
Define the matrix
\begin{align*}
X_{1} & =X_{0}+tQ_{1}.
\end{align*}
Let $\mfp$ be any of the ideals $\mfp_{1},\mfp_{2},\dots,\mfp_{\nu}$.
Let $\alpha_{\mfp},\beta_{\mfp},t_{\mfp}$ denote the images of $\alpha$,
$\beta$ and $t$ in $R/\mfp$, respectively. As before, let $y_{\mfp}$
denote the image of $y$ in $R/\mfp$. If we let
\[
L_{\mfp}=\begin{pmatrix}1 & \beta_{\mfp}\alpha_{\mfp}^{-1} & y_{\mfp}\beta_{\mfp}\alpha_{\mfp}^{-1}\\
0 & 1 & 0\\
0 & 0 & 1
\end{pmatrix}\oplus1_{n-3}\in\M_{n}(R/\mfp),
\]
then direct verification shows that $L_{\mfp}(X_{1})_{\mfp}L_{\mfp}^{-1}=\alpha_{\mfp}t_{\mfp}E_{11}\oplus W_{\mfp}$,
where $W_{\mfp}$ is the matrix in (\ref{eq:Wp}). Since $W_{\mfp}$
is regular and neither of its eigenvalues $0$ or $-y_{\mfp}$ equals
$\alpha_{\mfp}t_{\mfp}$ by (\ref{eq:at+y}), the matrix $\alpha_{\mfp}t_{\mfp}E_{11}\oplus W_{\mfp}$,
and hence $(X_{1})_{\mfp}\in\M_{n}(R/\mfp)$, is regular. We thus
see that $(X_{1})_{\mfp_{i}}$ is regular for all $i=1,\dots,\nu$.

By (\ref{eq:X0-Q0}) we have
\[
mA_{0}=[X_{0},Q_{0}]=[X_{0},Q_{1}]=[X_{1},Q_{1}],
\]
and since $(X_{1})_{\mfp_{i}}$ is regular and $m\in\mfp_{i}$ for
all $i=1,\dots,\nu$, we get $Q_{1}=g_{i}(X_{1})+p_{i}Q{}_{1}^{(i)}$,
for some $g_{i}(T)\in R[T]$ and $Q{}_{1}^{(i)}\in\M_{n}(R)$. Thus
\[
mA_{0}=[X_{1},g_{i}(X_{1})+p_{i}Q{}_{1}^{(i)}]=p_{i}[X_{1},Q{}_{1}^{(i)}],
\]
and so $mp_{i}^{-1}A_{0}=[X_{1},Q{}_{1}^{(i)}]$. Repeating the argument
if necessary, we obtain $mp_{i}^{-e_{i}}A_{0}\in[X_{1},\M_{n}(R)]$.
Running through each $i=1,\dots,\nu$ we obtain $A_{0}=[X_{1},Y]$
for some $Y\in\M_{n}(R)$, and hence $A=[N^{-1}X_{1}N,NYN^{-1}]$.
\end{proof}
By a theorem of Hungerford \cite{Hungerford} every principal ideal
ring (PIR) is a finite product of rings, each of which is a homomorphic
image of a PID. Together with Theorem~\ref{thm:Main} this immediately
implies the following:
\begin{cor}
\label{cor:Coroll-Main}Let $R$ be a PIR (not necessarily an integral
domain) and let $A\in\M_{n}(R)$, $n\geq2$, be a matrix with trace
zero. Then $A=[X,Y]$ for some $X,Y\in\M_{n}(R)$.
\end{cor}
We end this section by proving a strengthened version of Theorem~\ref{thm:Main}
for $n=3$.
\begin{prop}
\label{prop:n3regX}Let $R$ be a PID and let $A\in\M_{3}(R)$ be
a matrix with trace zero. Then $A=[X,Y]$ for some $X,Y\in\M_{3}(R)$
such that $X_{\mfp}$ is regular for all $\mfp\in\Specm R$. \end{prop}
\begin{proof}
As in the proof of Theorem~\ref{thm:Main} we may assume that $A$
is in Laffey-Reams form. Define the matrix
\[
X=\begin{pmatrix}0 & 0 & 0\\
x & -y & 0\\
q & z & 0
\end{pmatrix}\in\M_{3}(R).
\]
The same argument as in the proof of Theorem~\ref{thm:Main} shows
that $\Tr(XA)=0$ implies that $\Tr(X^{r}A)=0$ for all $r\geq0$.
Let $a_{23}'\in R$ be such that $a_{23}=a_{12}a_{23}'$, and let
$d\in R$ be a generator of $(a_{12},c(A))$. The condition $\Tr(XA)=0$
is then equivalent to
\begin{align*}
x & =hc(A)d^{-1}-a_{23}'z\\
y & =ha_{12}d^{-1},
\end{align*}
for any $h\in R$. We claim that the system of equations
\begin{equation}
\begin{cases}
x=hc(A)d^{-1}-a_{23}'z\\
y=ha_{12}d^{-1}\\
xz+qy=1
\end{cases}\label{eq:n3-system}
\end{equation}
has a solution in $x,y,q,z,h\in R$. Indeed, substituting the first
two equations in the last, we get
\[
-a_{23}'z^{2}+h(c(A)d^{-1}z+qa_{12}d^{-1})=1,
\]
and since $(c(A)d^{-1},a_{12}d^{-1})=(1)$ we can choose $z$ and
$q$ in $R$ such that $ $$c(A)d^{-1}z+qa_{12}d^{-1}=1$, and it
then remains to take
\[
h=1+a_{23}'z^{2}.
\]
Suppose now that $x,y,q,z,h\in R$ is a solution of (\ref{eq:n3-system}),
and let $\mfp\in\Specm R$. We show that $X_{\mfp}$ is regular. The
characteristic polynomial of $X$ is
\[
\lambda^{2}(\lambda+y)\in R[\lambda].
\]
We have
\[
X^{2}=\begin{pmatrix}0 & 0 & 0\\
-xy & y^{2} & 0\\
xz & -yz & 0
\end{pmatrix}.
\]
Thus, if $y\notin\mfp$ then $(X_{\mfp})^{2}\neq0$, and if $y\in\mfp$,
then we must have $xz\not\in\mfp$, so $(X_{\mfp})^{2}\neq0$ also
in this case. Furthermore, since $xz+qy=1$ we have
\[
X(X+y)=E_{31}\neq0.
\]
Thus the minimal polynomial of $X_{\mfp}$ must equal the characteristic
polynomial, so $X_{\mfp}$ is regular. Since we have $\Tr(X^{r}A)=0$
for all $r\geq0$, Proposition~(\ref{prop:Criterion}) implies that
$A=[X,Y]$, for some $Y\in\M_{3}(R)$.
\end{proof}
We remark that while the matrix $X$ in the proof of the above proposition
is regular modulo every $\mfp\in\Specm R$, it is not necessarily
regular. Moreover, while for $n=4$ we can find an analogous matrix
\[
X=\begin{pmatrix}0 & 0 & 0 & 0\\
x & -y & 0 & 0\\
q & z & 0 & 0\\
0 & 0 & 1 & -y
\end{pmatrix}
\]
such that $\Tr(AX)=0$ and $xz+yq=1$, in this case the matrix $X_{\mfp}$
may fail to be regular for some $\mfp\in\Spec R$.

\section{\label{sec:Further-directions}Further directions}

If $R$ is a field or if $R$ is a PID and $n=2$, we have shown that
every $A\in\M_{n}(R)$ with trace zero can be written $A=[X,Y]$ where
$X,Y\in\M_{n}(R)$ and $X$ is regular. Our proof of Theorem~\ref{thm:Main}
shows that for any PID $R$, $n\geq2$ and every $A\in\M_{n}(R)$
with trace zero we have $A=[X,Y]$ for some $X,Y\in\M_{n}(R)$ where
$X_{\mfp}$ is regular for all but finitely many maximal ideals $\mfp$
of $R$. Moreover, Proposition~\ref{prop:n3regX} says that when
$n=3$ the matrix $X$ can be chosen such that $X_{\mfp}$ is regular
for all maximal ideals $\mfp$.
\begin{problem*}
For $n\geq4$ and $A=[X,Y]$, is it always possible to choose $X$
such that $X_{\mfp}$ is regular for all maximal ideals $\mfp$?
\end{problem*}
This problem is interesting insofar as a proof, if possible, would
be likely to yield a substantially simplified proof of Theorem~\ref{thm:Main}.

It is natural to ask for generalisations of Theorem~\ref{thm:Main}
to rings other than PIRs. We first mention some counter-examples.
It was shown by Lissner \cite{Lissner} that the analogue of Theorem~\ref{thm:Main}
fails when $n=2$ and $R=k[x,y,z]$, where $k$ is a field, and more
generally that for $R=k[x_{1},\dots,x_{2n-1}]$ there exist matrices
in $\M_{n}(R)$ with trace zero which are not commutators (see \cite[Theorem~5.4]{Lissner}).
Rosset and Rosset \cite[Lemma~1.1]{Rosset} gave a sufficient criterion
for a $2\times2$ trace zero matrix over any commutative ring not
to be a commutator. They showed however, that a Noetherian integral
domain cannot satisfy their criterion unless it has dimension at least
$3$. This means that their criterion is not an obstruction to a $2\times2$
trace zero matrix over a one or two-dimensional Noetherian domain
being a commutator. Still, if $R$ is the two-dimensional domain $\R[x,y,z]/(x^{2}+y^{2}+z^{2}-1)$
it can be shown that there exists a matrix in $\M_{2}(R)$ with trace
zero which is not a commutator (this example goes back to Kaplansky;
see \cite[Section~4, Example~1]{Swan/62}, \cite[p.~532]{Lissner-OPrings}
or \cite[Section~3]{Rosset}).

A ring $R$ is called an \emph{OP-ring} if for every $n\geq1$ every
vector in $\bigwedge^{n-1}R^{n}$ is decomposable, that is, of the
form $v_{1}\wedge\dots\wedge v_{n-1}$ for some $v_{i}\in R^{n}$.
This is equivalent to saying that every vector in $R^{n}$ is an outer
product (hence the acronym OP). The notion of OP-ring was introduced
in \cite{Lissner-OPrings}. In particular, for $n=3$ the condition
on $R$ of being an OP-ring is equivalent to the condition that every
trace zero matrix in $\M_{2}(R)$ is a commutator (see \cite[Section~3]{Lissner}).
It is known that every Dedekind domain is an OP-ring \cite[p.~534]{Lissner-OPrings}
and that every polynomial ring in one variable over a Dedekind domain
is an OP-ring \cite[Theorem~1.2]{Towber}. This prompts the following
problem:
\begin{problem*}
Let $R$ be a Dedekind domain and assume that $A\in\M_{n}(R)$, $n\geq2$,
has trace zero. Is it true that $A=[X,Y]$ for some $X,Y\in\M_{n}(R)$?
\end{problem*}
\noindent Since Dedekind domains are OP-rings the question has an
affirmative answer for $n=2$, and one could ask the same question
for any OP-ring. In the setting of matrices over a Dedekind domain
the methods we have used to prove Theorem~\ref{thm:Main} are of
little use because they rely crucially on the underlying ring being
both atomic and B\'ezout, which implies that it is a PID.

\bibliographystyle{alex}
\bibliography{alex,/home/purem/lfvx79/Dropbox/Local_TeX_Files/bibtex/bib/alex,masterbibliography11jan11}

\end{document}